\documentclass{amsart}

\usepackage{amssymb}

\usepackage{enumitem}

\theoremstyle{plain}
\newtheorem{theorem}{Theorem}
\newtheorem{lemma}[theorem]{Lemma}

\theoremstyle{remark}
\newtheorem{remark}[theorem]{Remark}

\newcommand{\away}{\text{away}}
\newcommand{\axes}{\text{axes}}
\newcommand{\id}{\text{id}}
\newcommand{\locus}{\text{locus}}
\newcommand{\sym}{\text{sym}}

\DeclareMathOperator{\dist}{dist}
\DeclareMathOperator{\GL}{GL}

\begin{document}

\title{Entire maps with rational preperiodic points and multipliers}

\author{Xavier Buff}
\address{Institut de Math\'{e}matiques de Toulouse, UMR 5219, Universit\'{e} de Toulouse, CNRS, UPS, F-31062 Toulouse Cedex 9, France}
\email{xavier.buff@math.univ-toulouse.fr}

\author{Igors Gorbovickis}
\address{Constructor University Bremen gGmbH, Campus Ring 1, 28759 Bremen, Germany}
\email{igorbovickis@constructor.university}

\author{Valentin Huguin}
\address{Constructor University Bremen gGmbH, Campus Ring 1, 28759 Bremen, Germany}
\email{vhuguin@constructor.university}

\thanks{The research of the second and third authors was supported by the German Research Foundation (DFG, project number 455038303).}


\begin{abstract}
Given a number field $\mathbb{K} \subset \mathbb{C}$ that is not contained in $\mathbb{R}$, we prove the existence of a dense set of entire maps $f \colon \mathbb{C} \rightarrow \mathbb{C}$ whose preperiodic points and multipliers all lie in $\mathbb{K}$. This contrasts with the case of rational maps. In addition, we show that there exists an escaping quadratic-like map that is not conjugate to an affine escaping quadratic-like map and whose multipliers all lie in $\mathbb{Q}$.
\end{abstract}

\maketitle

\section{Introduction}

This article is motivated by the study of the arithmetic properties of dynamical systems. These have been particularly investigated in the case of iterated rational maps. As an example, a classical result in arithmetic dynamics, due to Northcott, implies the following:

\begin{theorem}[{\cite[Theorem~3]{N1950}}]
\label{theorem:ratPreper}
Assume that $\mathbb{K}$ is a number field and $f \in \mathbb{K}(z)$ is a rational map of degree $d \geq 2$. Then $f$ has only finitely many preperiodic points in $\mathbb{K}$.
\end{theorem}

Following a conjecture made by Milnor in~\cite{M2006}, questions concerning rational maps with integer or rational multipliers have also been investigated. In particular, Ji and Xie recently proved Milnor's conjecture, showing that power maps, Chebyshev maps and Latt\`{e}s maps are the only rational maps whose multipliers all lie in the ring of integers of a given imaginary quadratic field (see~\cite[Theorem~1.12]{JX2023}). The third author later proved the following stronger result:

\begin{theorem}[{\cite[Main Theorem]{H2023}}]
\label{theorem:ratMult}
Assume that $\mathbb{K} \subset \mathbb{C}$ is a number field and $f \in \mathbb{C}(z)$ is a rational map of degree $d \geq 2$ whose multiplier at each cycle lies in $\mathbb{K}$. Then $f$ is a power map, a Chebyshev map or a Latt\`{e}s map.
\end{theorem}

In this article, we prove that Theorems~\ref{theorem:ratPreper} and~\ref{theorem:ratMult} do not hold for transcendental entire maps. Our method is to construct maps with desired properties recursively, by successive perturbations. Thus, we extend the result below, due to Green, to a dynamical setting.

\begin{theorem}[{\cite[Theorem~1]{G1939}}]
There exists a transcendental entire function $f \colon \mathbb{C} \rightarrow \mathbb{C}$ such that $f\left( \mathbb{Q}(i) \right) \subseteq \mathbb{Q}(i)$.
\end{theorem}

Given a number field $\mathbb{K} \subset \mathbb{C}$ that is not contained in $\mathbb{R}$, we prove the existence of transcendental entire maps whose preperiodic points and multipliers all lie in $\mathbb{K}$. In fact, our proof does not rely on the arithmetic properties of such number fields, but only on the fact that these are both countable and dense in $\mathbb{C}$. Thus, we show the more general result below.

Recall that a \emph{Fr\'{e}chet space} is a complete, metrizable, locally convex topological vector space. Equivalently, a Fr\'{e}chet space is a real vector space $\mathcal{F}$ equipped with a separating countable family of seminorms for which $\mathcal{F}$ is complete. For example, the set of all entire maps endowed with the topology of local uniform convergence on $\mathbb{C}$ forms a Fr\'{e}chet space. Also note that every Banach space is a Fr\'{e}chet space. Finally, given a set $F$ and topologies $\mathcal{T}_{1}, \mathcal{T}_{2}$ on $F$, recall that $\mathcal{T}_{1}$ is \emph{finer} than $\mathcal{T}_{2}$ if each open subset of $F$ with respect to $\mathcal{T}_{2}$ is also open with respect to $\mathcal{T}_{1}$.

\begin{theorem}
\label{theorem:entire}
Assume that $\mathcal{F}$ is a Fr\'{e}chet space of entire maps that contains all the complex polynomial maps and whose topology is finer than the topology of local uniform convergence on $\mathbb{C}$, $E$ is a countable and dense subset of $\mathbb{C}$ and $\Lambda$ is a dense subset of $\mathbb{C}$. Then the set of $f \in \mathcal{F}$ such that
\begin{itemize}
\item $f^{-1}(E) = E$,
\item the periodic points of $f$ all lie in $E$,
\item the multipliers of $f$ at its cycles all lie in $\Lambda$
\end{itemize}
is dense in $\mathcal{F}$.
\end{theorem}

Suppose now that $\mathbb{K} \subset \mathbb{C}$ is a number field that is not contained in $\mathbb{R}$. Letting $\mathcal{F}$ be the usual Fr\'{e}chet space of all entire maps and setting $E = \Lambda = \mathbb{K}$, Theorem~\ref{theorem:entire} shows that there exists a non-affine entire map $f \colon \mathbb{C} \rightarrow \mathbb{C}$ that is neither a power map nor a Chebyshev map and whose preperiodic points and multipliers all lie in $\mathbb{K}$; this map $f$ is necessarily transcendental by Theorem~\ref{theorem:ratMult}. Alternatively, taking $\mathcal{F}$ to be a Fr\'{e}chet space of entire maps that contains all the complex polynomials as a non-dense subset and whose topology is finer than the topology of local uniform convergence on $\mathbb{C}$, one can avoid invoking Theorem~\ref{theorem:ratMult} and directly use Theorem~\ref{theorem:entire} to prove the existence of a transcendental entire map $f \colon \mathbb{C} \rightarrow \mathbb{C}$ whose preperiodic points and multipliers all lie in $\mathbb{K}$. This shows that Theorems~\ref{theorem:ratPreper} and~\ref{theorem:ratMult} do not hold for general entire maps.

Our proof of Theorem~\ref{theorem:entire} consists in adjusting the dynamics of any given entire map by a succession of small perturbations. This method can also be adapted to prove analogues of Theorem~\ref{theorem:entire} with various additional symmetry conditions. As an illustration, we prove the result below, which treats the case of real and even entire maps.

\begin{theorem}
\label{theorem:realEven}
Assume that
\begin{itemize}
\item $\mathcal{F}$ is a Fr\'{e}chet space of real and even entire maps that contains all the real and even polynomial maps and whose topology is finer than the topology of local uniform convergence on $\mathbb{C}$,
\item $E$ is a countable and dense subset of $\mathbb{C}$ that is symmetric with respect to the real and imaginary axes and such that $E \cap (\mathbb{R} \cup i \mathbb{R})$ is dense in $\mathbb{R} \cup i \mathbb{R}$,
\item $\Lambda$ is a dense subset of $\mathbb{C}$ that is symmetric with respect to the real axis and such that $\Lambda \cap \mathbb{R}$ is dense in $\mathbb{R}$.
\end{itemize}
Then the set of $f \in \mathcal{F}$ such that
\begin{itemize}
\item $f^{-1}(E) = E$,
\item the periodic points of $f$ all lie in $E$,
\item the multipliers of $f$ at its cycles all lie in $\Lambda$
\end{itemize}
is dense in $\mathcal{F}$.
\end{theorem}

Note that, for every number field $\mathbb{K} \subset \mathbb{C}$ that is invariant under complex conjugation but not contained in $\mathbb{R}$, Theorem~\ref{theorem:realEven} implies the existence of real, even and transcendental entire maps whose preperiodic points and multipliers all lie in $\mathbb{K}$.

\begin{remark}
\label{remark:general}
In fact, our proofs show that we can replace $\Lambda$ in Theorems~\ref{theorem:entire} and~\ref{theorem:realEven} by a sequence $\left( \Lambda_{n} \right)_{n \geq 1}$ of subsets of $\mathbb{C}$ with the same properties and ask instead that the multipliers of $f$ at its cycles with period $n$ all lie in $\Lambda_{n}$ for all $n \in \mathbb{Z}_{\geq 1}$.
\end{remark}

\begin{remark}
Our proofs of Theorems~\ref{theorem:entire} and~\ref{theorem:realEven} crucially use the assumption that the topology of the Fr\'{e}chet space $\mathcal{F}$ of entire maps is finer than the topology of local uniform convergence on $\mathbb{C}$. We note that there exist Fr\'{e}chet spaces of entire maps that do not have this property. One can construct such Fr\'{e}chet spaces as follows: Denote by $F$ the vector space of all entire maps and by $\mathcal{F}_{0}$ the usual Fr\'{e}chet space of all entire maps, which has the topology of local uniform convergence. For every $\Phi \in \GL(F)$, there exists a unique Fr\'{e}chet space $\mathcal{F}_{\Phi}$ whose underlying vector space is $F$ and such that $\Phi \colon \mathcal{F}_{0} \rightarrow \mathcal{F}_{\Phi}$ is a homeomorphism. Choose a sequence $\left( f_{n} \right)_{n \geq 0}$ of elements of $F$ that converges to $f \in F \setminus S$ in $\mathcal{F}_{0}$ and choose $g \in F \setminus S$ different from $f$, where $S$ denotes the vector subspace of $F$ spanned by $\left\lbrace f_{n} : n \geq 0 \right\rbrace$. Then there exists $\Phi \in \GL(F)$ such that $\Phi\left( f_{n} \right) = f_{n}$ for all $n \geq 0$ and $\Phi(f) = g$. Setting $\mathcal{F} = \mathcal{F}_{\Phi}$, the sequence $\left( f_{n} \right)_{n \geq 0}$ converges to $g$ in $\mathcal{F}$, and in particular the topology of $\mathcal{F}$ is not finer than the topology of local uniform convergence on $\mathbb{C}$.
\end{remark}

Finally, we apply Theorem~\ref{theorem:realEven} in order to prove the existence of a transcendental entire map with various properties, including that of having an escaping quadratic-like map restriction with rational multipliers. An \emph{escaping quadratic-like map} is a holomorphic covering map $f \colon U \rightarrow V$ of degree $2$, where $U, V$ are nonempty open subsets of $\mathbb{C}$ such that $U \Subset V$ and $V$ is simply connected. In this case, note that $U$ has precisely two connected components, which are mapped biholomorphically onto $V$ by $f$. We say that $f$ is \emph{conjugate to an affine escaping quadratic-like map} if there exists a univalent map $\varphi \colon V \rightarrow \mathbb{C}$ such that $\varphi \circ f \circ \varphi^{-1}$ is affine on each of the two connected components of $\varphi(U)$. In order to present a modified version of Ji and Xie's proof of~\cite[Theorem~1.12]{JX2023}, which generalizes Milnor's conjecture about rational maps with integer multipliers, the first and third authors together with Gauthier and Raissy proved the following:

\begin{theorem}[{\cite[Proposition~14]{BGHR2023}}]
\label{theorem:entMult}
Assume that $\mathcal{O}_{\mathbb{K}}$ is the ring of integers of an imaginary quadratic field $\mathbb{K}$ and $f \colon U \rightarrow V$ is an escaping quadratic-like map whose multiplier at each cycle lies in $\mathcal{O}_{\mathbb{K}}$. Then $f$ is conjugate to an affine escaping quadratic-like map.
\end{theorem}

In light of Theorem~\ref{theorem:entMult}, one may ask whether every escaping quadratic-like map that has only rational multipliers is conjugate to an affine escaping quadratic-like map. This question was the initial motivation for our study. A positive answer to this question would provide an alternative proof of Theorem~\ref{theorem:ratMult}. However, we prove here that the answer is negative. More precisely, as an application of Theorem~\ref{theorem:realEven} and Remark~\ref{remark:general}, we obtain the result below. Here, $\mathbb{D} \subset \mathbb{C}$ denotes the unit disk.

\begin{theorem}
\label{theorem:quadLike}
There exists a real, even and transcendental entire map $f \colon \mathbb{C} \rightarrow \mathbb{C}$ such that
\begin{itemize}
\item $f \colon \mathbb{R} \rightarrow \mathbb{R}$ is convex,
\item $f(\mathbb{Q}) \subset \mathbb{Q}$,
\item the periodic points of $f \colon f^{-1}(\mathbb{D}) \cap \mathbb{D} \rightarrow \mathbb{D}$ and $f \colon \mathbb{R} \rightarrow \mathbb{R}$ coincide and all lie in $\mathbb{Q}$, and the multipliers of $f$ at these periodic points all lie in $\mathbb{Q}$,
\item $f \colon f^{-1}(\mathbb{D}) \cap \mathbb{D} \rightarrow \mathbb{D}$ is an escaping quadratic-like map that is not conjugate to an affine escaping quadratic-like map.
\end{itemize}
\end{theorem}

\begin{remark}
Assume here that $\mathbb{K} \subset \mathbb{C}$ is a number field that is not contained in $\mathbb{R}$. Then one can directly make use of Theorem~\ref{theorem:entire} to prove the existence of an escaping quadratic-like map that is not conjugate to an affine escaping quadratic-like map and whose multipliers all lie in $\mathbb{K}$. More precisely, one may proceed as follows: As already explained, Theorem~\ref{theorem:entire} implies the existence of a transcendental entire map $f \colon \mathbb{C} \rightarrow \mathbb{C}$ whose multipliers all lie in $\mathbb{K}$. Then there exist open subsets $U, V$ of $\mathbb{C}$ and $n \in \mathbb{Z}_{\geq 1}$ such that $f^{\circ n} \colon U \rightarrow V$ is a well-defined escaping quadratic-like map (see~\cite[Proposition~B.3]{B2000}, which follows from the Ahlfors five islands theorem). Note that the multipliers of $f^{\circ n} \colon U \rightarrow V$ all lie in $\mathbb{K}$. Furthermore, $f^{\circ n} \colon U \rightarrow V$ is not conjugate to an affine escaping quadratic-like map since $f \colon \mathbb{C} \rightarrow \mathbb{C}$ is neither a power map nor a Chebyshev map (see~\cite[Lemma~12]{BGHR2023}, which follows from the results of~\cite{R1922}).
\end{remark}

After writing this article, the first and third authors together with Gauthier and Raissy realized that Theorem~\ref{theorem:entMult} also yields the analogue of Milnor's conjecture for entire maps. Thus, they obtained the result below about entire maps with integer multipliers, which also contrasts with Theorem~\ref{theorem:entire}.

\begin{theorem}[{\cite[Theorem~2]{BGHR2023}}]
Assume that $\mathcal{O}_{\mathbb{K}}$ is the ring of integers of an imaginary quadratic field $\mathbb{K} \subset \mathbb{C}$ and $f \colon \mathbb{C} \rightarrow \mathbb{C}$ is a non-affine entire map whose multiplier at each cycle lies in $\mathcal{O}_{\mathbb{K}}$. Then $f$ is a power map or a Chebyshev map.
\end{theorem}

Ji, Xie and Zhang also later generalized Theorem~\ref{theorem:ratMult}, showing that power maps, Chebyshev maps and Latt\`{e}s maps are the only rational maps whose multipliers all have a modulus in a given number field (see~\cite[Theorem~1.4]{JXZ2023}).

In Section~\ref{section:entire}, we prove Theorem~\ref{theorem:entire} by perturbative arguments. In Section~\ref{section:realEven}, we adapt our proof of Theorem~\ref{theorem:entire} in order to prove Theorem~\ref{theorem:realEven}. In Section~\ref{section:quadLike}, we prove Theorem~\ref{theorem:quadLike} by applying Theorem~\ref{theorem:realEven} and Remark~\ref{remark:general} in a particular setting.

\section{Proof of Theorem~\ref{theorem:entire}}
\label{section:entire}

Throughout this section, we fix a Fr\'{e}chet space $\mathcal{F}$ of entire maps that contains all the complex polynomial maps and whose topology is finer than the topology of local uniform convergence on $\mathbb{C}$. We denote by $\left( \lVert . \rVert_{j} \right)_{j \geq 0}$ a sequence of seminorms associated to $\mathcal{F}$, and we define the distance $d_{\mathcal{F}}$ on $\mathcal{F}$ by \[ d_{\mathcal{F}}(f, g) = \sum_{j = 0}^{+\infty} 2^{-j} \min\left\lbrace 1, \lVert f -g \rVert_{j} \right\rbrace \, \text{,} \] which induces the topology of $\mathcal{F}$ and makes it a complete metric space.

Throughout this section, we also fix a countable and dense subset $E$ of $\mathbb{C}$ and a dense subset $\Lambda$ of $\mathbb{C}$. Removing $1$ from $\Lambda$ if necessary, we assume that $\Lambda \subseteq \mathbb{C} \setminus \lbrace 1 \rbrace$.

We define \[ \mathcal{E} = \lbrace z \mapsto a z +b : a, b \in \mathbb{C} \rbrace \] to be the set of affine maps on $\mathbb{C}$, which forms a vector space of real dimension $4$, and hence a closed subset of $\mathcal{F}$ with empty interior.

Given $f \in \mathcal{F}$ and a subset $A$ of $\mathbb{C}$, we define \[ \mathcal{F}(f, A) = \left\lbrace g \in \mathcal{F} : g \vert_{A} = f \vert_{A} \text{ and } g^{\prime} \vert_{A} = f^{\prime} \vert_{A} \right\rbrace \, \text{,} \] which is closed in $\mathcal{F}$ because the topology of $\mathcal{F}$ is finer than the topology of local uniform convergence. Given $f \in \mathcal{F}$ and subsets $A, C$ of $\mathbb{C}$, we define \[ \mathcal{F}(f, A, C) = \left\lbrace g \in \mathcal{F} : g \vert_{A \cup C} = f \vert_{A \cup C} \text{ and } g^{\prime} \vert_{C} = f^{\prime} \vert_{C} \right\rbrace \, \text{,} \] which is also closed in $\mathcal{F}$ for the same reason.

We shall now show Theorem~\ref{theorem:entire}. Our proof proceeds roughly as follows: We note that an entire map $f \colon \mathbb{C} \rightarrow \mathbb{C}$ satisfies the conditions of Theorem~\ref{theorem:entire} if and only if it satisfies countably many conditions, each one referring to values and derivatives of $f$ at finitely many points. Now, suppose that $g \in \mathcal{F}$ satisfies finitely many of these conditions. Then there exist finite subsets $A, C$ of $\mathbb{C}$ such that each $h \in \mathcal{F}(g, A, C)$ sufficiently close to $g$ satisfies the same conditions as $g$. Adding well-chosen small polynomials to $g$, we prove that one can find $h \in \mathcal{F}(g, A, C)$ arbitrarily close to $g$ that also satisfy additional conditions from the aforementioned countable list. We use this process as a recursive step. Thus, given any $f_{0} \in \mathcal{F}$ and any $\varepsilon \in \mathbb{R}_{> 0}$, we construct a Cauchy sequence $\left( f_{n} \right)_{n \geq 0}$ of elements of $\mathcal{F}$ whose limit $f \in \mathcal{F}$ satisfies the conditions of Theorem~\ref{theorem:entire} and also $d_{\mathcal{F}}\left( f_{0}, f \right) < \varepsilon$. This will complete our proof of Theorem~\ref{theorem:entire}.

\subsection{Polynomial interpolation}

Let us exhibit here families of complex polynomial maps that satisfy certain conditions on their values and derivatives. We will use them to adjust the dynamics of an entire map by perturbation.

\begin{lemma}
\label{lemma:entireVal}
For every finite subset $A$ of $\mathbb{C}$, every $b \in \mathbb{C} \setminus A$ and every $\zeta \in \mathbb{C}$, there exists a complex polynomial map $P_{A, b, \zeta} \colon \mathbb{C} \rightarrow \mathbb{C}$ of degree at most $2 \lvert A \rvert$ such that \[ P_{A, b, \zeta} \vert_{A} = 0 \, \text{,} \quad P_{A, b, \zeta}^{\prime} \vert_{A} = 0 \quad \text{and} \quad P_{A, b, \zeta}(b) = \zeta \, \text{.} \] Furthermore, \[ \lim_{\zeta \rightarrow 0} \sup\left\lbrace \left\lVert P_{A, b, \zeta} \right\rVert : \lvert A \rvert = N, \, \dist(b, A) \geq r, \, A \subset D(0, R) \right\rbrace = 0 \] for all $N \in \mathbb{Z}_{\geq 0}$, all $r, R \in \mathbb{R}_{> 0}$ and all seminorms $\lVert . \rVert$ on $\mathcal{F}$.
\end{lemma}

\begin{proof}
Given a finite set $A \subset \mathbb{C}$, $b \in \mathbb{C} \setminus A$ and $\zeta \in \mathbb{C}$, define $P_{A, b, \zeta} \colon \mathbb{C} \rightarrow \mathbb{C}$ by \[ P_{A, b, \zeta}(z) = \zeta \prod_{a \in A} \left( \frac{z -a}{b -a} \right)^{2} \, \text{.} \] Then the required conditions are satisfied.
\end{proof}

\begin{lemma}
\label{lemma:entireDiff}
For every finite subset $A$ of $\mathbb{C}$, every $b \in \mathbb{C} \setminus A$ and every $\lambda \in \mathbb{C}$, there exists a complex polynomial map $Q_{A, b, \lambda} \colon \mathbb{C} \rightarrow \mathbb{C}$ of degree at most $2 \lvert A \rvert +1$ such that \[ Q_{A, b, \lambda} \vert_{A} = 0 \, \text{,} \quad Q_{A, b, \lambda}^{\prime} \vert_{A} = 0 \, \text{,} \quad Q_{A, b, \lambda}(b) = 0 \quad \text{and} \quad Q_{A, b, \lambda}^{\prime}(b) = \lambda \, \text{.} \] Furthermore, $\lim\limits_{\lambda \rightarrow 0} \left\lVert Q_{A, b, \lambda} \right\rVert = 0$ for all finite subsets $A$ of $\mathbb{C}$, all $b \in \mathbb{C} \setminus A$ and all seminorms $\lVert . \rVert$ on $\mathcal{F}$.
\end{lemma}

\begin{proof}
Given a finite set $A \subset \mathbb{C}$, $b \in \mathbb{C} \setminus A$ and $\lambda \in \mathbb{C}$, define $Q_{A, b, \lambda} \colon \mathbb{C} \rightarrow \mathbb{C}$ by \[ Q_{A, b, \lambda}(z) = \lambda (z -b) \prod_{a \in A} \left( \frac{z -a}{b -a} \right)^{2} \, \text{.} \] Then the required conditions are satisfied.
\end{proof}

\subsection{Adjustment of images}

Let us use here Lemma~\ref{lemma:entireVal} to adjust the image of a point under an entire map by perturbation.

\begin{lemma}
\label{lemma:entireImage}
Suppose that $f \in \mathcal{F}$, $A$ is a finite subset of $f^{-1}(E)$ and $b \in \mathbb{C}$. Then, for every $\varepsilon \in \mathbb{R}_{> 0}$, there exists $g \in \mathcal{F}(f, A)$ such that $d_{\mathcal{F}}(f, g) < \varepsilon$ and $g(b) \in E$.
\end{lemma}

\begin{proof}
If $b \in A$, setting $g = f$, the required conditions are satisfied. Now, suppose that $b \in \mathbb{C} \setminus A$. For $\zeta \in \mathbb{C}$, define \[ g_{\zeta} = f +P_{A, b, \zeta -f(b)} \in \mathcal{F}(f, A) \, \text{.} \] Then $\lim\limits_{\zeta \rightarrow f(b)} g_{\zeta} = f$ in $\mathcal{F}$. Therefore, since $E$ is dense in $\mathbb{C}$, there exists $\zeta \in E$ such that $d_{\mathcal{F}}\left( f, g_{\zeta} \right) < \varepsilon$. Setting $g = g_{\zeta}$, the required conditions are satisfied. Thus, the lemma is proved.
\end{proof}

\subsection{Adjustment of preimages}

Let us use Lemma~\ref{lemma:entireVal} to adjust the preimages in a disk of finitely many points under an entire map.

\begin{lemma}
\label{lemma:entirePreim}
Suppose that $f \in \mathcal{F} \setminus \mathcal{E}$, $A$ is a finite subset of $E$, $B$ is a finite subset of $\mathbb{C}$ and $R \in \mathbb{R}_{> 0}$. Then, for every $\varepsilon \in \mathbb{R}_{> 0}$, there exists $g \in \mathcal{F}(f, A)$ such that $d_{\mathcal{F}}(f, g) < \varepsilon$ and $g^{-1}(B) \cap D(0, R) \subset E$.
\end{lemma}

\begin{proof}
As $f$ is not constant, there exists $S > R$ such that $f^{-1}(B) \cap \partial D(0, S) = \varnothing$. Denote by $N \in \mathbb{Z}_{\geq 0}$ the number of preimages in $D(0, S)$ of the elements of $B$ under $f$, counting multiplicities. As the topology of $\mathcal{F}$ is finer than the topology of local uniform convergence, reducing $\varepsilon$ if necessary, we may assume that \[ \forall g \in \mathcal{F}, \, d_{\mathcal{F}}(f, g) < \varepsilon \Longrightarrow \sup_{\partial D(0, S)} \lvert f -g \rvert < \dist\left( f\left( \partial D(0, S) \right), B \right) \, \text{,} \] so that the elements of $B$ have together exactly $N$ preimages in $D(0, S)$ under any $g \in \mathcal{F}$ such that $d_{\mathcal{F}}(f, g) < \varepsilon$, counting multiplicities. Given $g \in \mathcal{F}$, denote by
\begin{itemize}
\item $m_{g}$ the number of preimages in $D(0, R) \cap E$ of the elements of $B$ under $g$, not counting multiplicities,
\item $n_{g}$ the number of preimages in $D(0, R) \cap E$ of the elements of $B$ under $g$, counting multiplicities,
\item $N_{g}$ the total number of preimages in $D(0, R)$ of the elements of $B$ under $g$, counting multiplicities.
\end{itemize}
We shall prove that there exists $g \in \mathcal{F}(f, A)$ such that $d_{\mathcal{F}}(f, g) < \varepsilon$ and $n_{g} = N_{g}$. Now, note that $m_{g} \leq N$ for all $g \in \mathcal{F}$ such that $d_{\mathcal{F}}(f, g) < \varepsilon$. Therefore, it suffices to prove that, if $g \in \mathcal{F}(f, A)$ satisfies $d_{\mathcal{F}}(f, g) < \varepsilon$ and $n_{g} < N_{g}$, then there exists $h \in \mathcal{F}(f, A)$ such that $d_{\mathcal{F}}(f, h) < \varepsilon$ and $m_{h} > m_{g}$. Thus, suppose that $g \in \mathcal{F}(f, A)$ is such a map. Define \[ A^{\prime} = g^{-1}(B) \cap D(0, R) \cap E \, \text{.} \] As $n_{g} < N_{g}$, there exists $w \in D(0, R) \setminus E$ such that $g(w) \in B$. For $\zeta \in \mathbb{C} \setminus \left( A \cup A^{\prime} \right)$, define \[ h_{\zeta} = g +P_{A \cup A^{\prime}, \zeta, g(w) -g(\zeta)} \in \mathcal{F}\left( g, A \cup A^{\prime} \right) \, \text{.} \] Then $\lim\limits_{\zeta \rightarrow w} h_{\zeta} = g$ in $\mathcal{F}$ because $w \in \mathbb{C} \setminus \left( A \cup A^{\prime} \right)$. Therefore, since $E$ is dense in $\mathbb{C}$ and $d_{\mathcal{F}}(f, g) < \varepsilon$, there exists $\zeta \in \left( D(0, R) \cap E \right) \setminus \left( A \cup A^{\prime} \right)$ such that $d_{\mathcal{F}}\left( f, h_{\zeta} \right) < \varepsilon$. Setting $h = h_{\zeta}$, we have $m_{h} > m_{g}$ since $h(\zeta) = g(w) \in B$ and $h \in \mathcal{F}\left( g, A^{\prime} \right)$. Thus, we have shown that there exists $g \in \mathcal{F}(f, A)$ such that $d_{\mathcal{F}}(f, g) < \varepsilon$ and $n_{g} = N_{g}$, which completes the proof of the lemma.
\end{proof}

\subsection{Adjustment of cycles}

Now, let us use Lemmas~\ref{lemma:entireVal} and~\ref{lemma:entireDiff} in order to adjust the positions and multipliers of cycles for an entire map.

Given $p \in \mathbb{Z}_{\geq 1}$ and $R \in \mathbb{R}_{> 0}$, we say that $f \in \mathcal{F}$ has the property $\left( \Gamma_{p, R} \right)$ if the cycles for $f$ with period at most $p$ that intersect $D(0, R)$ are all contained in $E$ and their multipliers all lie in $\Lambda$.

\begin{lemma}
\label{lemma:entireCycle}
Suppose that $f \in \mathcal{F} \setminus \mathcal{E}$, $A$ is a finite subset of $E \cap f^{-1}(E)$, $C$ is a finite union of cycles for $f$ that are contained in $E$ and whose multipliers lie in $\Lambda$, $p \in \mathbb{Z}_{\geq 1}$ and $R \in \mathbb{R}_{> 0}$. Then, for every $\varepsilon \in \mathbb{R}_{> 0}$, there exists $g \in \mathcal{F}(f, A, C)$ that has the property $\left( \Gamma_{p, R} \right)$ and satisfies $d_{\mathcal{F}}(f, g) < \varepsilon$.
\end{lemma}

\begin{proof}
As $f \colon \mathbb{C} \rightarrow \mathbb{C}$ is not injective, we have $f^{\circ j} \neq \id_{\mathbb{C}}$ for all $j \in \mathbb{Z}_{\geq 1}$, and hence there exists $S > R$ such that $\partial D(0, S)$ contains no periodic point for $f$ with period at most $p$. Denote by $N \in \mathbb{Z}_{\geq 0}$ the number of periodic points for $f$ in $D(0, S)$ with period at most $p$, counting multiplicities. Since the topology of $\mathcal{F}$ is finer than the topology of local uniform convergence, reducing the number $\varepsilon$ if necessary, we may assume that any $g \in \mathcal{F}$ that satisfies $d_{\mathcal{F}}(f, g) < \varepsilon$ has exactly $N$ periodic points in $D(0, S)$ with period at most $p$, counting multiplicities. Given $g \in \mathcal{F}$, denote by
\begin{itemize}
\item $n_{g}$ the number of periodic points for $g$ in $D(0, R)$ with period at most $p$ whose cycle is contained in $E$ and whose multiplier lies in $\Lambda$, not counting multiplicities,
\item $N_{g}$ the number of periodic points for $g$ in $D(0, R)$ with period at most $p$, counting multiplicities.
\end{itemize}
Let us prove that there exists $g \in \mathcal{F}(f, A, C)$ such that $d_{\mathcal{F}}(f, g) < \varepsilon$ and $n_{g} = N_{g}$. Now, note that $n_{g} \leq N$ for all $g \in \mathcal{F}$ such that $d_{\mathcal{F}}(f, g) < \varepsilon$. Therefore, it suffices to show that, if $g \in \mathcal{F}(f, A, C)$ satisfies $d_{\mathcal{F}}(f, g) < \varepsilon$ and $n_{g} < N_{g}$, then there exists $h \in \mathcal{F}(f, A, C)$ such that $d_{\mathcal{F}}(f, h) < \varepsilon$ and $n_{h} > n_{g}$. Suppose that $g \in \mathcal{F}(f, A, C)$ is such a map. Define $C^{\prime}$ to be the union of the cycles for $g$ with period at most $p$ that intersect $D(0, R)$, are contained in $E$ and whose multipliers lie in $\Lambda$. Because $\Lambda \subseteq \mathbb{C} \setminus \lbrace 1 \rbrace$, the elements of $C^{\prime}$ are all simple periodic points for $g$. Therefore, as $n_{g} < N_{g}$, there is a cycle $\Omega$ for $g$ with period at most $p$ that intersects $D(0, R)$ but is not contained in $C^{\prime}$. Define \[ \boldsymbol{Z} = \left\lbrace \boldsymbol{\zeta} \in \mathbb{C}^{\Omega} : \boldsymbol{\zeta} \vert_{\Omega \cap E} = \id_{\Omega \cap E} \right\rbrace \, \text{.} \] For $\boldsymbol{\zeta} \in \boldsymbol{Z}$, set $\Omega_{\boldsymbol{\zeta}} = \left\lbrace \boldsymbol{\zeta}(\omega) : \omega \in \Omega \right\rbrace$. Since $A \subseteq E \cap g^{-1}(E)$, $\Omega \subset \mathbb{C} \setminus \left( C \cup C^{\prime} \right)$ and $\Omega \cap D(0, R) \neq \varnothing$, there is a neighborhood $V$ of $\id_{\Omega}$ in $\boldsymbol{Z}$ such that, for each $\boldsymbol{\zeta} \in V$,
\begin{itemize}
\item $\boldsymbol{\zeta}(\omega) \in \mathbb{C} \setminus A$ for all $\omega \in \Omega \setminus \left( E \cap g^{-1}(E) \right)$,
\item $\Omega_{\boldsymbol{\zeta}} \subset \mathbb{C} \setminus \left( C \cup C^{\prime} \right)$,
\item $\boldsymbol{\zeta}(\omega) \neq \boldsymbol{\zeta}\left( \omega^{\prime} \right)$ for all distinct $\omega, \omega^{\prime} \in \Omega$,
\item $\Omega_{\boldsymbol{\zeta}} \cap D(0, R) \neq \varnothing$.
\end{itemize}
For $\boldsymbol{\zeta} \in V$, define \[ g_{\boldsymbol{\zeta}} = g +\sum_{\omega \in \Omega \setminus \left( E \cap g^{-1}(E) \right)} \phi_{\boldsymbol{\zeta}, \omega} \in \mathcal{F}\left( g, A \cup C \cup C^{\prime} \right) \, \text{,} \] where, for every $\omega \in \Omega \setminus \left( E \cap g^{-1}(E) \right)$, \[ \phi_{\boldsymbol{\zeta}, \omega} = P_{A \cup C \cup C^{\prime} \cup \left( \Omega_{\boldsymbol{\zeta}} \setminus \left\lbrace \boldsymbol{\zeta}(\omega) \right\rbrace \right), \boldsymbol{\zeta}(\omega), \boldsymbol{\zeta}\left( g(\omega) \right) -g\left( \boldsymbol{\zeta}(\omega) \right)} \, \text{.} \] For each $\boldsymbol{\zeta} \in V$, we have $g_{\boldsymbol{\zeta}}\left( \boldsymbol{\zeta}(\omega) \right) = \boldsymbol{\zeta}\left( g(\omega) \right)$ for all $\omega \in \Omega$, and hence $\Omega_{\boldsymbol{\zeta}}$ forms a cycle for $g_{\boldsymbol{\zeta}}$. Now, note that $\lim\limits_{\boldsymbol{\zeta} \rightarrow \id_{\Omega}} g_{\boldsymbol{\zeta}} = g$ in $\mathcal{F}$. Therefore, as $E$ is dense in $\mathbb{C}$ and $d_{\mathcal{F}}(f, g) < \varepsilon$, there exists $\boldsymbol{\zeta} \in V \cap E^{\Omega}$ such that $d_{\mathcal{F}}\left( f, g_{\boldsymbol{\zeta}} \right) < \varepsilon$. For $\boldsymbol{\lambda} \in \mathbb{C}^{\Omega}$, define \[ h_{\boldsymbol{\lambda}} = g_{\boldsymbol{\zeta}} +\sum_{\omega \in \Omega} \psi_{\boldsymbol{\lambda}, \omega} \in \mathcal{F}\left( g_{\boldsymbol{\zeta}}, A \cup \Omega_{\boldsymbol{\zeta}}, C \cup C^{\prime} \right) \, \text{,} \] where, for every $\omega \in \Omega$, \[ \psi_{\boldsymbol{\lambda}, \omega} = Q_{C \cup C^{\prime} \cup \left( \left( A \cup \Omega_{\boldsymbol{\zeta}} \right) \setminus \left\lbrace \boldsymbol{\zeta}(\omega) \right\rbrace \right), \boldsymbol{\zeta}(\omega), \boldsymbol{\lambda}(\omega)} \, \text{.} \] Then $\Omega_{\boldsymbol{\zeta}}$ is a cycle for $h_{\boldsymbol{\lambda}}$ and we have $h_{\boldsymbol{\lambda}}^{\prime}\left( \boldsymbol{\zeta}(\omega) \right) = g_{\boldsymbol{\zeta}}^{\prime}\left( \boldsymbol{\zeta}(\omega) \right) +\boldsymbol{\lambda}(\omega)$ for all $\omega \in \Omega$ and all $\boldsymbol{\lambda} \in \mathbb{C}^{\Omega}$. Moreover, $\lim\limits_{\boldsymbol{\lambda} \rightarrow 0} h_{\boldsymbol{\lambda}} = g_{\boldsymbol{\zeta}}$ in $\mathcal{F}$. Therefore, since $\Lambda$ is dense in $\mathbb{C}$ and $d_{\mathcal{F}}\left( f, g_{\boldsymbol{\zeta}} \right) < \varepsilon$, there exists $\boldsymbol{\lambda} \in \mathbb{C}^{\Omega}$ such that $d_{\mathcal{F}}\left( f, h_{\boldsymbol{\lambda}} \right) < \varepsilon$ and the multiplier of $h_{\boldsymbol{\lambda}}$ at $\Omega_{\boldsymbol{\zeta}}$ lies in $\Lambda$. Setting $h = h_{\boldsymbol{\lambda}}$, we have $n_{h} > n_{g}$. Thus, we have proved that there exists $g \in \mathcal{F}(f, A, C)$ such that $d_{\mathcal{F}}(f, g) < \varepsilon$ and $n_{g} = N_{g}$, which completes the proof of the lemma.
\end{proof}

\subsection{Proof of the theorem}

Finally, let us combine here Lemmas~\ref{lemma:entireImage}, \ref{lemma:entirePreim} and~\ref{lemma:entireCycle} in order to prove Theorem~\ref{theorem:entire}.

\begin{proof}[Proof of Theorem~\ref{theorem:entire}]
Suppose that $f_{0} \in \mathcal{F}$ and $\varepsilon \in \mathbb{R}_{> 0}$. We shall show that there exists $f \in \mathcal{F}$ such that
\begin{itemize}
\item $d_{\mathcal{F}}\left( f_{0}, f \right) < \varepsilon$,
\item $f^{-1}(E) = E$,
\item the cycles for $f$ are all contained in $E$ and their multipliers all lie in $\Lambda$.
\end{itemize}
Since $\mathcal{E}$ is a closed subset of $\mathcal{F}$ with empty interior, replacing $f_{0}$ and $\varepsilon$ if necessary, we may assume that $f_{0} \in \mathcal{F} \setminus \mathcal{E}$ and $\varepsilon \in \left( 0, \dist\left( f_{0}, \mathcal{E} \right) \right)$. Write $E = \left\lbrace e_{j} : j \in \mathbb{Z}_{\geq 1} \right\rbrace$, with $e_{j} \neq e_{k}$ for all distinct $j, k \in \mathbb{Z}_{\geq 1}$. For $n \in \mathbb{Z}_{\geq 1}$, define $E_{n} = \left\lbrace e_{1}, \dotsc, e_{n} \right\rbrace$. Let us show that there exists a sequence $\left( f_{n} \right)_{n \geq 0}$ of elements of $\mathcal{F}$ such that, for every $n \geq 1$,
\begin{enumerate}
\item\label{item:entireDist} $d_{\mathcal{F}}\left( f_{n -1}, f_{n} \right) < \frac{\varepsilon}{2^{n}}$,
\item\label{item:entireEqual} $f_{n} \in \mathcal{F}\left( f_{n -1}, A_{n -1}, C_{n -1} \right)$ (see the definitions below),
\item\label{item:entireImage} $f_{n}\left( e_{n} \right) \in E$,
\item\label{item:entirePreim} $f_{n}^{-1}\left( E_{n} \right) \cap D(0, n) \subset E$,
\item\label{item:entireCycle} $f_{n}$ has the property $\left( \Gamma_{n, n} \right)$,
\end{enumerate}
where, for every $n \geq 0$, \[ A_{n} = E_{n} \cup \left( f_{n}^{-1}\left( E_{n} \right) \cap D(0, n) \right) \] and $C_{n}$ denotes the union of the cycles for $f_{n}$ with period at most $n$ that intersect $D(0, n)$, with $A_{0} = C_{0} = \varnothing$ by convention. Let us proceed by recursion. Suppose that $n \in \mathbb{Z}_{\geq 0}$ and $f_{1}, \dotsc, f_{n} \in \mathcal{F}$ satisfy these conditions~\eqref{item:entireDist}--\eqref{item:entireCycle}, and let us prove the existence of $f_{n +1} \in \mathcal{F}$ that satisfies the same conditions. By the conditions~\eqref{item:entireEqual}, \eqref{item:entireImage} and~\eqref{item:entireCycle}, we have $A_{n} \cup C_{n} \subset f_{n}^{-1}(E)$. By Lemma~\ref{lemma:entireImage}, it follows that there exists $g_{n} \in \mathcal{F}$ such that \[ g_{n} \in \mathcal{F}\left( f_{n}, A_{n} \cup C_{n} \right) \, \text{,} \quad d_{\mathcal{F}}\left( f_{n}, g_{n} \right) < \frac{\varepsilon}{3 \cdot 2^{n +1}} \, \text{,} \quad g_{n}\left( e_{n +1} \right) \in E \, \text{.} \] Note that $g_{n} \in \mathcal{F} \setminus \mathcal{E}$ since $d_{\mathcal{F}}\left( f_{0}, g_{n} \right) < \varepsilon$ by the condition~\eqref{item:entireDist}. Now, define \[ A_{n}^{\prime} = A_{n} \cup \left\lbrace e_{n +1} \right\rbrace = E_{n +1} \cup \left( f_{n}^{-1}\left( E_{n} \right) \cap D(0, n) \right) \, \text{.} \] Then $A_{n}^{\prime} \subseteq E \cap g_{n}^{-1}(E)$ by the conditions~\eqref{item:entireEqual}, \eqref{item:entireImage} and~\eqref{item:entirePreim}. Moreover, $C_{n}$ is a union of cycles for $g_{n}$ that are all contained in $E$ and whose multipliers all lie in $\Lambda$ by the condition~\eqref{item:entireCycle}. Therefore, by Lemma~\ref{lemma:entireCycle}, there exists $h_{n} \in \mathcal{F}$ such that \[ h_{n} \in \mathcal{F}\left( g_{n}, A_{n}^{\prime}, C_{n} \right) \, \text{,} \quad d_{\mathcal{F}}\left( g_{n}, h_{n} \right) < \frac{\varepsilon}{3 \cdot 2^{n +1}} \, \text{,} \quad h_{n} \text{ has the property } \left( \Gamma_{n +1, n +2} \right) \, \text{.} \] Note that $h_{n} \in \mathcal{F} \setminus \mathcal{E}$ since $d_{\mathcal{F}}\left( f_{0}, h_{n} \right) < \varepsilon$ by the condition~\eqref{item:entireDist}. In particular, we have $h_{n}^{\circ j} \neq \id_{\mathbb{C}}$ for all $j \in \lbrace 1, \dotsc, n +1 \rbrace$, and hence there exists $R_{n} \in (n +1, n +2)$ such that $\partial D\left( 0, R_{n} \right)$ contains no periodic point for $h_{n}$ with period at most $n +1$. Now, denote by $X_{n}$ the union of the cycles for $h_{n}$ with period at most $n +1$ that intersect $D\left( 0, R_{n} \right)$. Then the cycles for $h_{n}$ in $X_{n}$ are all contained in $E$ and their multipliers all lie in $\Lambda$. Denote by $N_{n} \in \mathbb{Z}_{\geq 0}$ the number of periodic points for $h_{n}$ in $D\left( 0, R_{n} \right)$ with period at most $n +1$, counting multiplicities. Since $\Lambda \subseteq \mathbb{C} \setminus \lbrace 1 \rbrace$, the elements of $X_{n}$ are all simple periodic points for $h_{n}$, and hence $N_{n}$ equals the cardinality of $X_{n} \cap D\left( 0, R_{n} \right)$. Moreover, since the topology of $\mathcal{F}$ is finer than the topology of local uniform convergence, there exists $\varepsilon_{n} \in \left( 0, \frac{\varepsilon}{3 \cdot 2^{n +1}} \right)$ such that any $k_{n} \in \mathcal{F}$ such that $d_{\mathcal{F}}\left( h_{n}, k_{n} \right) < \varepsilon_{n}$ has exactly $N_{n}$ periodic points in $D\left( 0, R_{n} \right)$ with period at most $n +1$, counting multiplicities. Since $A_{n}^{\prime} \cup X_{n} \subset E$, it follows from Lemma~\ref{lemma:entirePreim} that there exists $f_{n +1} \in \mathcal{F}$ such that \[ f_{n +1} \in \mathcal{F}\left( h_{n}, A_{n}^{\prime} \cup X_{n} \right) \, \text{,} \quad d_{\mathcal{F}}\left( h_{n}, f_{n +1} \right) < \varepsilon_{n} \, \text{,} \quad f_{n +1}^{-1}\left( E_{n +1} \right) \cap D(0, n +1) \subset E \, \text{.} \] Then $f_{n +1}$ clearly satisfies the conditions~\eqref{item:entireDist}--\eqref{item:entirePreim}. Moreover, $X_{n}$ is also a union of cycles for $f_{n +1}$ and $f_{n +1}$ has exactly $N_{n}$ periodic points in $D\left( 0, R_{n} \right)$ with period at most $n +1$, counting multiplicities. Therefore, as $X_{n} \cap D\left( 0, R_{n} \right)$ has cardinality $N_{n}$, the cycles for $f_{n +1}$ with period at most $n +1$ that intersect $D\left( 0, R_{n} \right)$ are the cycles contained in $X_{n}$. Since these are all contained in $E$ and their multipliers all lie in $\Lambda$, the map $f_{n +1}$ also satisfies the condition~\eqref{item:entireCycle} because $R_{n} > n +1$. Thus, we have proved the existence of such a sequence $\left( f_{n} \right)_{n \geq 0}$ of elements of $\mathcal{F}$.

It follows from the condition~\eqref{item:entireDist} that $\left( f_{n} \right)_{n \geq 0}$ is a Cauchy sequence in $\left( \mathcal{F}, d_{\mathcal{F}} \right)$, and its limit $f \in \mathcal{F}$ satisfies $d_{\mathcal{F}}\left( f_{0}, f \right) < \varepsilon$. In particular, $f \in \mathcal{F} \setminus \mathcal{E}$. Let us prove that $f$ satisfies the required conditions. It suffices to show that, for each $m \in \mathbb{Z}_{\geq 1}$,
\begin{enumerate}[label=\textup{(\roman*)}, ref=\textup{\roman*}]
\item\label{item:entireImageBis} $f\left( e_{m} \right) \in E$,
\item\label{item:entirePreimBis} $f^{-1}\left( E_{m} \right) \cap D(0, m) \subset E$,
\item\label{item:entireCycleBis} $f$ has the property $\left( \Gamma_{m, m} \right)$.
\end{enumerate}
Suppose that $m \in \mathbb{Z}_{\geq 1}$. Then $f_{n}\left( e_{m} \right) = f_{m}\left( e_{m} \right)$ for all $n \geq m$ by the condition~\eqref{item:entireEqual}, and hence $f\left( e_{m} \right) = f_{m}\left( e_{m} \right)$ since $\left( f_{n} \right)_{n \geq 0}$ converges pointwise to $f$. Therefore, by the condition~\eqref{item:entireImage}, the condition~\eqref{item:entireImageBis} is satisfied. By the condition~\eqref{item:entireEqual} and as $\left( f_{n} \right)_{n \geq 0}$ converges pointwise to $f$, for every $n \geq m$, \[ f_{n}^{-1}\left( E_{m} \right) \cap D(0, m) \subseteq f_{n +1}^{-1}\left( E_{m} \right) \cap D(0, m) \subseteq f^{-1}\left( E_{m} \right) \cap D(0, m) \, \text{.} \] Therefore, as $f^{-1}\left( E_{m} \right) \cap D(0, m)$ is finite since $f \in \mathcal{F} \setminus \mathcal{E}$, there exists $N \in \mathbb{Z}_{\geq m}$ such that \[ \bigcup_{n \geq N} \left( f_{n}^{-1}\left( E_{m} \right) \cap D(0, m) \right) = f_{N}^{-1}\left( E_{m} \right) \cap D(0, m) \, \text{.} \] As $\left( f_{n} \right)_{n \geq 0}$ converges locally uniformly to $f$ and $f \in \mathcal{F} \setminus \mathcal{E}$, it follows that \[ f^{-1}\left( E_{m} \right) \cap D(0, m) \subseteq \overline{\bigcup_{n \geq N} \left( f_{n}^{-1}\left( E_{m} \right) \cap D(0, m) \right)} = f_{N}^{-1}\left( E_{m} \right) \cap D(0, m) \, \text{.} \] Therefore, by the condition~\eqref{item:entirePreim}, the condition~\eqref{item:entirePreimBis} is satisfied. For $n \in \mathbb{Z}_{\geq m}$, define $Y_{n} \subseteq C_{n}$ to be the union of the cycles for $f_{n}$ with period at most $m$ that intersect $D(0, m)$. Also define $Y$ to be the union of the cycles for $f$ with period at most $m$ that intersect $D(0, m)$. By the condition~\eqref{item:entireEqual} and since $\left( f_{n} \right)_{n \geq 0}$ converges pointwise to $f$, for each $n \geq m$, we have $Y_{n} \subseteq Y_{n +1} \subseteq Y$. Therefore, since $Y$ is finite because $f \in \mathcal{F} \setminus \mathcal{E}$, there exists $N \in \mathbb{Z}_{\geq m}$ such that $\bigcup\limits_{n \geq N} Y_{n} = Y_{N}$. Since $\left( f_{n} \right)_{n \geq 0}$ converges locally uniformly to $f$ and $f \in \mathcal{F} \setminus \mathcal{E}$, it follows that \[ Y \subseteq \overline{\bigcup_{n \geq N} Y_{n}} = Y_{N} \, \text{.} \] Moreover, $f \in \mathcal{F}\left( f_{N}, Y_{N} \right)$ by the condition~\eqref{item:entireEqual} and since $\left( f_{n} \right)_{n \geq 0}$ converges locally uniformly to $f$. Therefore, by the condition~\eqref{item:entireCycle}, the condition~\eqref{item:entireCycleBis} is also satisfied. This completes the proof of the theorem.
\end{proof}

\section{Proof of Theorem~\ref{theorem:realEven}}
\label{section:realEven}

We shall adapt here our proof of Theorem~\ref{theorem:entire} in order to prove Theorem~\ref{theorem:realEven}.

We fix a Fr\'{e}chet space $\mathcal{F}$ of real and even entire maps that contains all the real and even polynomial maps and whose topology is finer than the topology of local uniform convergence on $\mathbb{C}$. As in Section~\ref{section:entire}, taking a countable family $\left( \lVert . \rVert_{j} \right)_{j \geq 0}$ of seminorms associated to $\mathcal{F}$, we define the distance $d_{\mathcal{F}}$ on $\mathcal{F}$ by \[ d_{\mathcal{F}}(f, g) = \sum_{j = 0}^{+\infty} 2^{-j} \min\left\lbrace 1, \lVert f -g \rVert_{j} \right\rbrace \, \text{.} \]

We also fix a countable and dense subset $E$ of $\mathbb{C}$ that is symmetric with respect to both the real and imaginary axes and such that $E \cap (\mathbb{R} \cup i \mathbb{R})$ is dense in $\mathbb{R} \cup i \mathbb{R}$ and a dense subset $\Lambda$ of $\mathbb{C}$ that is symmetric with respect to the real axis and such that $\Lambda \cap \mathbb{R}$ is dense in $\mathbb{R}$. Furthermore, removing $1$ from $\Lambda$ if necessary, we assume that $\Lambda \subseteq \mathbb{C} \setminus \lbrace 1 \rbrace$.

We define \[ \mathcal{E} = \lbrace z \mapsto a : a \in \mathbb{R} \rbrace \] to be the set of real constant maps on $\mathbb{C}$, which equals the set of all real and even affine maps on $\mathbb{C}$. The set $\mathcal{E}$ forms a real vector space of dimension $1$, and hence a closed subset of $\mathcal{F}$ with empty interior.

As in Section~\ref{section:entire}, given $f \in \mathcal{F}$ and subsets $A, C$ of $\mathbb{C}$, we define \[ \mathcal{F}(f, A) = \left\lbrace g \in \mathcal{F} : g \vert_{A} = f \vert_{A} \text{ and } g^{\prime} \vert_{A} = f^{\prime} \vert_{A} \right\rbrace \] and \[ \mathcal{F}(f, A, C) = \left\lbrace g \in \mathcal{F} : g \vert_{A \cup C} = f \vert_{A \cup C} \text{ and } g^{\prime} \vert_{C} = f^{\prime} \vert_{C} \right\rbrace \, \text{.} \]

Given a subset $A$ of $\mathbb{C}$, we define \[ A^{\sym} = A \cup \sigma(A) \cup (-A) \cup \left( -\sigma(A) \right) \, \text{,} \] where $\sigma \colon \mathbb{C} \rightarrow \mathbb{C}$ is the complex conjugation, so that $A^{\sym}$ is the smallest subset of $\mathbb{C}$ that contains $A$ and is symmetric with respect to the real and imaginary axes.

Now, to prove Theorem~\ref{theorem:realEven}, let us follow and adapt the strategy used in Section~\ref{section:entire}. The differences with our proof of Theorem~\ref{theorem:entire} arise from the fact that any real and even entire map has $0$ as a critical point and maps $\mathbb{R} \cup i \mathbb{R}$ into $\mathbb{R}$.

\subsection{Polynomial interpolation}

First, let us adapt Lemmas~\ref{lemma:entireVal} and~\ref{lemma:entireDiff} in order to work with real and even polynomial maps.

\begin{lemma}
\hangindent\leftmargini
$\bullet$\hskip\labelsep For each finite set $A \subset \mathbb{C}$ such that $A^{\sym} = A$, each $b \in (\mathbb{R} \cup i \mathbb{R}) \setminus A$ and each $\zeta \in \mathbb{R}$, there exists a real and even polynomial map $P_{A, b, \zeta}^{\axes} \colon \mathbb{C} \rightarrow \mathbb{C}$ of degree at most $2 \lvert A \rvert$ such that \[ P_{A, b, \zeta}^{\axes} \vert_{A} = 0 \, \text{,} \quad \left( P_{A, b, \zeta}^{\axes} \right)^{\prime} \vert_{A} = 0 \quad \text{and} \quad P_{A, b, \zeta}^{\axes}(b) = \zeta \, \text{.} \] Moreover, \[ \lim_{\zeta \rightarrow 0} \sup\left\lbrace \left\lVert P_{A, b, \zeta}^{\axes} \right\rVert : \lvert A \rvert = N, \, \dist(b, A) \geq r, \, A \subset D(0, R) \right\rbrace = 0 \] for all $N \in \mathbb{Z}_{\geq 0}$, all $r, R \in \mathbb{R}_{> 0}$ and all seminorms $\lVert . \rVert$ on $\mathcal{F}$.
\begin{itemize}
\item For each finite set $A \subset \mathbb{C}$ such that $A^{\sym} = A$, each $b \in \mathbb{C} \setminus (\mathbb{R} \cup i \mathbb{R} \cup A)$ and each $\zeta \in \mathbb{C}$, there exists a real and even polynomial map $P_{A, b, \zeta}^{\away} \colon \mathbb{C} \rightarrow \mathbb{C}$ of degree at most $2 \lvert A \rvert +2$ such that \[ P_{A, b, \zeta}^{\away} \vert_{A} = 0 \, \text{,} \quad \left( P_{A, b, \zeta}^{\away} \right)^{\prime} \vert_{A} = 0 \quad \text{and} \quad P_{A, b, \zeta}^{\away}(b) = \zeta \, \text{.} \] Moreover, \[ \lim_{\zeta \rightarrow 0} \sup\left\lbrace \left\lVert P_{A, b, \zeta}^{\away} \right\rVert : \lvert A \rvert = N, \, \dist(b, \mathbb{R} \cup i \mathbb{R} \cup A) \geq r, \, A \subset D(0, R) \right\rbrace = 0 \] for all $N \in \mathbb{Z}_{\geq 0}$, all $r, R \in \mathbb{R}_{> 0}$ and all seminorms $\lVert . \rVert$ on $\mathcal{F}$.
\end{itemize}
\end{lemma}

\begin{proof}
Suppose that $A \subset \mathbb{C}$ is a finite set such that $A^{\sym} = A$. For $b \in (\mathbb{R} \cup i \mathbb{R}) \setminus A$ and $\zeta \in \mathbb{R}$, define $P_{A, b, \zeta}^{\axes} \colon \mathbb{C} \rightarrow \mathbb{C}$ by \[ P_{A, b, \zeta}^{\axes}(z) = \zeta \prod_{a \in A} \left( \frac{z -a}{b -a} \right)^{2} \, \text{.} \] For $b \in \mathbb{C} \setminus (\mathbb{R} \cup i \mathbb{R} \cup A)$ and $\zeta \in \mathbb{C}$, define $P_{A, b, \zeta}^{\away} \colon \mathbb{C} \rightarrow \mathbb{C}$ by \[ P_{A, b, \zeta}^{\away}(z) = \Gamma_{A, b, \zeta}(z) \prod_{a \in A} (z -a)^{2} \, \text{,} \] where \[ \Gamma_{A, b, \zeta}(z) = \frac{\Im(\gamma) z^{2} +\Re(\gamma) \Im\left( b^{2} \right) -\Im(\gamma) \Re\left( b^{2} \right)}{\Im\left( b^{2} \right)} \quad \text{and} \quad \gamma = \frac{\zeta}{\prod\limits_{a \in A} (b -a)^{2}} \, \text{.} \] Then the required conditions are satisfied.
\end{proof}

\begin{lemma}
\hangindent\leftmargini
$\bullet$\hskip\labelsep For each finite set $A \subset \mathbb{C}$ such that $A^{\sym} = A$, each $b \in \mathbb{R}^{*} \setminus A$ and each $\lambda \in \mathbb{R}$, there exists a real and even polynomial map $Q_{A, b, \lambda}^{\axes} \colon \mathbb{C} \rightarrow \mathbb{C}$ of degree at most $2 \lvert A \rvert +2$ such that \[ Q_{A, b, \lambda}^{\axes} \vert_{A} = 0 \, \text{,} \quad \left( Q_{A, b, \lambda}^{\axes} \right)^{\prime} \vert_{A} = 0 \, \text{,} \quad Q_{A, b, \lambda}^{\axes}(b) = 0 \quad \text{and} \quad \left( Q_{A, b, \lambda}^{\axes} \right)^{\prime}(b) = \lambda \, \text{.} \] Moreover, $\lim\limits_{\lambda \rightarrow 0} \left\lVert Q_{A, b, \lambda}^{\axes} \right\rVert = 0$ for all finite sets $A \subset \mathbb{C}$ such that $A^{\sym} = A$, all $b \in \mathbb{R}^{*} \setminus A$ and all seminorms $\lVert . \rVert$ on $\mathcal{F}$.
\begin{itemize}
\item For each finite set $A \subset \mathbb{C}$ such that $A^{\sym} = A$, each $b \in \mathbb{C} \setminus (\mathbb{R} \cup i \mathbb{R} \cup A)$ and each $\lambda \in \mathbb{C}$, there exists a real and even polynomial map $Q_{A, b, \lambda}^{\away} \colon \mathbb{C} \rightarrow \mathbb{C}$ of degree at most $2 \lvert A \rvert +6$ such that \[ Q_{A, b, \lambda}^{\away} \vert_{A} = 0 \, \text{,} \quad \left( Q_{A, b, \lambda}^{\away} \right)^{\prime} \vert_{A} = 0 \, \text{,} \quad Q_{A, b, \lambda}^{\away}(b) = 0 \quad \text{and} \quad \left( Q_{A, b, \lambda}^{\away} \right)^{\prime}(b) = \lambda \, \text{.} \] Moreover, $\lim\limits_{\lambda \rightarrow 0} \left\lVert Q_{A, b, \lambda}^{\away} \right\rVert = 0$ for all finite sets $A \subset \mathbb{C}$ such that $A^{\sym} = A$, all $b \in \mathbb{C} \setminus (\mathbb{R} \cup i \mathbb{R} \cup A)$ and all seminorms $\lVert . \rVert$ on $\mathcal{F}$.
\end{itemize}
\end{lemma}

\begin{proof}
Suppose that $A \subset \mathbb{C}$ is a finite set such that $A^{\sym} = A$. For $b \in \mathbb{R}^{*} \setminus A$ and $\lambda \in \mathbb{R}$, define $Q_{A, b, \lambda}^{\axes} \colon \mathbb{C} \rightarrow \mathbb{C}$ by \[ Q_{A, b, \lambda}^{\axes}(z) = \lambda \left( \frac{z^{2} -b^{2}}{2 b} \right) \prod_{a \in A} \left( \frac{z -a}{b -a} \right)^{2} \, \text{.} \] For $b \in \mathbb{C} \setminus (\mathbb{R} \cup i \mathbb{R} \cup A)$ and $\lambda \in \mathbb{C}$, define $Q_{A, b, \lambda}^{\away} \colon \mathbb{C} \rightarrow \mathbb{C}$ by \[ Q_{A, b, \lambda}^{\away}(z) = \Delta_{A, b, \lambda}(z) \left( z^{4} -2 \Re\left( b^{2} \right) z^{2} +\lvert b \rvert^{4} \right) \prod_{a \in A} (z -a)^{2} \, \text{,} \] where \[ \Delta_{A, b, \lambda}(z) = \frac{\Im(\delta) z^{2} +\Re(\delta) \Im\left( b^{2} \right) -\Im(\delta) \Re\left( b^{2} \right)}{\Im\left( b^{2} \right)} \quad \text{and} \quad \delta = \frac{\lambda}{4 i b \Im\left( b^{2} \right) \prod\limits_{a \in A} (b -a)^{2}} \, \text{.} \] Then the required conditions are satisfied.
\end{proof}

\subsection{Adjustment of images}

Let us adapt here Lemma~\ref{lemma:entireImage} in the setting of real and even entire maps.

\begin{lemma}
\label{lemma:realEvenImage}
Suppose that $f \in \mathcal{F}$, $A$ is a finite subset of $\mathbb{C}$ and $b \in \mathbb{C}$. Also assume that $A \subset f^{-1}(E)$ or $b \in \mathbb{C} \setminus A^{\sym}$. Then, for every $\varepsilon \in \mathbb{R}_{> 0}$, there exists $g \in \mathcal{F}(f, A)$ such that $d_{\mathcal{F}}(f, g) < \varepsilon$ and $g(b) \in E$.
\end{lemma}

\begin{proof}
Because $f$ is real and even and $E^{\sym} = E$, replacing $A$ by $A^{\sym}$ if necessary, we may assume that $A^{\sym} = A$. If $A \subset f^{-1}(E)$ and $b \in A$, then $g = f$ satisfies the required conditions. Now, suppose that $b \in \mathbb{C} \setminus A$. Define \[ \locus = \begin{cases} \axes & \text{if } b \in \mathbb{R} \cup i \mathbb{R}\\ \away & \text{if } b \in \mathbb{C} \setminus (\mathbb{R} \cup i \mathbb{R}) \end{cases} \, \text{,} \quad \mathbb{C}^{\locus} = \begin{cases} \mathbb{R} & \text{if } \locus = \axes\\ \mathbb{C} & \text{if } \locus = \away \end{cases} \, \text{.} \] Note that $f(b) \in \mathbb{C}^{\locus}$. Therefore, for $\zeta \in \mathbb{C}^{\locus}$, we can define \[ g_{\zeta} = f +P_{A, b, \zeta -f(b)}^{\locus} \in \mathcal{F}(f, A) \, \text{.} \] Then $\lim\limits_{\zeta \rightarrow f(b)} g_{\zeta} = f$ in $\mathcal{F}$. Therefore, since $E \cap \mathbb{C}^{\locus}$ is dense in $\mathbb{C}^{\locus}$, there exists $\zeta \in E \cap \mathbb{C}^{\locus}$ such that $d_{\mathcal{F}}\left( f, g_{\zeta} \right) < \varepsilon$. Setting $g = g_{\zeta}$, the required conditions are satisfied. Thus, the lemma is proved.
\end{proof}

\subsection{Adjustment of preimages}

Let us present here an analogue of Lemma~\ref{lemma:entirePreim}.

\begin{lemma}
\label{lemma:realEvenPreim}
Suppose that $f \in \mathcal{F} \setminus \mathcal{E}$, $A$ is a finite subset of $E$, $B$ is a finite subset of $\mathbb{C}$ and $R \in \mathbb{R}_{> 0}$. Then, for every $\varepsilon \in \mathbb{R}_{> 0}$, there exists $g \in \mathcal{F}(f, A)$ such that $d_{\mathcal{F}}(f, g) < \varepsilon$ and $g^{-1}(B) \cap D(0, R) \subset E$.
\end{lemma}

\begin{proof}
Since $E^{\sym} = E$, replacing $A$ by $A^{\sym}$ and $B$ by $B^{\sym}$ if necessary, we may assume that $A^{\sym} = A$ and $B^{\sym} = B$. Given $g \in \mathcal{F}$, denote by
\begin{itemize}
\item $m_{g}$ the number of preimages in $D(0, R) \cap E$ of the elements of $B$ under $g$, not counting multiplicities,
\item $n_{g}$ the number of preimages in $D(0, R) \cap E$ of the elements of $B$ under $g$, counting multiplicities,
\item $N_{g}$ the total number of preimages in $D(0, R)$ of the elements of $B$ under $g$, counting multiplicities.
\end{itemize}
We shall prove that there exists $g \in \mathcal{F}(f, A)$ such that $d_{\mathcal{F}}(f, g) < \varepsilon$ and $n_{g} = N_{g}$. Using the same preliminary arguments as in the proof of Lemma~\ref{lemma:entirePreim}, reducing $\varepsilon$ if necessary, we may assume that $m_{g} \leq N$ for each $g \in \mathcal{F}$ such that $d_{\mathcal{F}}(f, g) < \varepsilon$, for some bound $N \in \mathbb{Z}_{\geq 0}$ independent of $g \in \mathcal{F}$. Therefore, it suffices to show that, if $g \in \mathcal{F}(f, A)$ satisfies $d_{\mathcal{F}}(f, g) < \varepsilon$ and $n_{g} < N_{g}$, then there exists $h \in \mathcal{F}(f, A)$ such that $d_{\mathcal{F}}(f, h) < \varepsilon$ and $m_{h} > m_{g}$. Suppose that $g \in \mathcal{F}(f, A)$ is such a map. Define \[ A^{\prime} = g^{-1}(B) \cap D(0, R) \cap E \, \text{,} \] which is finite of cardinality $m_{g}$ and such that $\left( A^{\prime} \right)^{\sym} = A^{\prime}$. Since $n_{g} < N_{g}$, there exists $w \in D(0, R) \setminus E$ such that $g(w) \in B$. Now, define \[ \locus = \begin{cases} \axes & \text{if } w \in \mathbb{R} \cup i \mathbb{R}\\ \away & \text{if } w \in \mathbb{C} \setminus (\mathbb{R} \cup i \mathbb{R}) \end{cases} \, \text{,} \quad \mathbb{C}^{\locus} = \begin{cases} \mathbb{R} \cup i \mathbb{R} & \text{if } \locus = \axes\\ \mathbb{C} \setminus (\mathbb{R} \cup i \mathbb{R}) & \text{if } \locus = \away \end{cases} \, \text{.} \] For $\zeta \in \mathbb{C}^{\locus} \setminus \left( A \cup A^{\prime} \right)$, define \[ h_{\zeta} = g +P_{A \cup A^{\prime}, \zeta, g(w) -g(\zeta)}^{\locus} \in \mathcal{F}\left( g, A \cup A^{\prime} \right) \, \text{.} \] Then $\lim\limits_{\zeta \rightarrow w} h_{\zeta} = g$ in $\mathcal{F}$ because $w \in \mathbb{C}^{\locus} \setminus \left( A \cup A^{\prime} \right)$. Therefore, since $E \cap \mathbb{C}^{\locus}$ is dense in $\mathbb{C}^{\locus}$ and $d_{\mathcal{F}}(f, g) < \varepsilon$, there exists $\zeta \in \left( D(0, R) \cap E \cap \mathbb{C}^{\locus} \right) \setminus \left( A \cup A^{\prime} \right)$ such that $d_{\mathcal{F}}\left( f, h_{\zeta} \right) < \varepsilon$. Setting $h = h_{\zeta}$, we have $m_{h} > m_{g}$ since $h(\zeta) = g(w) \in B$ and $h \in \mathcal{F}\left( g, A^{\prime} \right)$. This completes the proof of the lemma.
\end{proof}

\subsection{Adjustment of cycles}

Now, let us adapt Lemma~\ref{lemma:entireCycle} in the current setting. Note that, if $f \in \mathcal{F}$ has $0$ as a periodic point, then the multiplier of $f$ at $0$ equals $0$ since $f$ is even. Thus, as $0$ need not lie in $\Lambda$ by assumption, the point $0$ requires special treatment.

As in Section~\ref{section:entire}, for $p \in \mathbb{Z}_{\geq 1}$ and $R \in \mathbb{R}_{> 0}$, we say that $f \in \mathcal{F}$ has the property $\left( \Gamma_{p, R} \right)$ if its cycles with period at most $p$ that intersect $D(0, R)$ are all contained in $E$ and their multipliers all lie in $\Lambda$.

\begin{lemma}
\label{lemma:realEvenCycle}
Suppose that $f \in \mathcal{F} \setminus \mathcal{E}$, $A$ is a finite subset of $E \cap f^{-1}(E)$, $C$ is a finite union of cycles for $f$ that are contained in $E$ and whose multipliers lie in $\Lambda$, $p \in \mathbb{Z}_{\geq 1}$ and $R \in \mathbb{R}_{> 0}$. Also assume that $0 \in \mathbb{C} \setminus E$ or $0$ is not periodic for $f$ with period at most $p$. Then, for every $\varepsilon \in \mathbb{R}_{> 0}$, there exists $g \in \mathcal{F}(f, A, C)$ that has the property $\left( \Gamma_{p, R} \right)$ and satisfies $d_{\mathcal{F}}(f, g) < \varepsilon$.
\end{lemma}

\begin{proof}
As $f$ is real and even and $E^{\sym} = E$, replacing $A$ by $A^{\sym}$ if necessary, we may assume that $A^{\sym} = A$. Since the topology of $\mathcal{F}$ is finer than the topology of local uniform convergence, reducing the number $\varepsilon$ if necessary, we may also assume that $0 \in \mathbb{C} \setminus E$ or $0$ is not periodic with period at most $p$ for any $g \in \mathcal{F}$ such that $d_{\mathcal{F}}(f, g) < \varepsilon$. Given $g \in \mathcal{F}$, denote by
\begin{itemize}
\item $n_{g}$ the number of periodic points for $g$ in $D(0, R)$ with period at most $p$ whose cycle is contained in $E$ and whose multiplier lies in $\Lambda$, not counting multiplicities,
\item $N_{g}$ the number of periodic points for $g$ in $D(0, R)$ with period at most $p$, counting multiplicities.
\end{itemize}
Let us prove that there exists $g \in \mathcal{F}(f, A, C)$ such that $d_{\mathcal{F}}(f, g) < \varepsilon$ and $n_{g} = N_{g}$. Using the same preliminary arguments as in the proof of Lemma~\ref{lemma:entireCycle}, reducing $\varepsilon$ if necessary, we may assume that $n_{g} \leq N$ for each $g \in \mathcal{F}$ such that $d_{\mathcal{F}}(f, g) < \varepsilon$, for some bound $N \in \mathbb{Z}_{\geq 0}$ independent of $g \in \mathcal{F}$. Therefore, it suffices to show that, if $g \in \mathcal{F}(f, A, C)$ satisfies $d_{\mathcal{F}}(f, g) < \varepsilon$ and $n_{g} < N_{g}$, then there exists $h \in \mathcal{F}(f, A, C)$ such that $d_{\mathcal{F}}(f, h) < \varepsilon$ and $n_{h} > n_{g}$. Thus, suppose that $g \in \mathcal{F}(f, A, C)$ is such a map. Now, denote by $C^{\prime}$ the union of the cycles for $g$ with period at most $p$ that intersect $D(0, R)$, are contained in $E$ and whose multipliers lie in $\Lambda$. The elements of $C^{\prime}$ are all simple periodic points for $g$ since $\Lambda \subseteq \mathbb{C} \setminus \lbrace 1 \rbrace$. Therefore, as $n_{g} < N_{g}$, there exists a cycle $\Omega$ for $g$ with period at most $p$ that intersects $D(0, R)$ but not $C^{\prime}$. We have $g(\mathbb{R} \cup i \mathbb{R}) \subseteq \mathbb{R}$ because $g$ is real and even, and hence either $\Omega \subset \mathbb{R}$ or $\Omega \subset \mathbb{C} \setminus (\mathbb{R} \cup i \mathbb{R})$. Now, define \[ \locus = \begin{cases} \axes & \text{if } \Omega \subset \mathbb{R}\\ \away & \text{if } \Omega \subset \mathbb{C} \setminus (\mathbb{R} \cup i \mathbb{R}) \end{cases} \, \text{,} \quad \mathbb{C}^{\locus} = \begin{cases} \mathbb{R} & \text{if } \locus = \axes\\ \mathbb{C} \setminus (\mathbb{R} \cup i \mathbb{R}) & \text{if } \locus = \away \end{cases} \, \text{.} \] Denote by $\sigma \colon \mathbb{C} \rightarrow \mathbb{C}$ the complex conjugation. Since $g$ is real, $\sigma(\Omega)$ is also a cycle for $g$, and in particular either $\Omega = \sigma(\Omega)$ or $\Omega \cap \sigma(\Omega) = \varnothing$. Define \[ \boldsymbol{Z} = \left\lbrace \boldsymbol{\zeta} \in \left( \mathbb{C}^{\locus} \right)^{\Omega} : \boldsymbol{\zeta} \vert_{\Omega \cap E} = \id_{\Omega \cap E} \text{ and } \boldsymbol{\zeta} \circ \sigma \vert_{\Omega \cap \sigma(\Omega)} = \sigma \circ \boldsymbol{\zeta} \vert_{\Omega \cap \sigma(\Omega)} \right\rbrace \, \text{.} \] For $\boldsymbol{\zeta} \in \boldsymbol{Z}$, set $\Omega_{\boldsymbol{\zeta}} = \left\lbrace \boldsymbol{\zeta}(\omega) : \omega \in \Omega \right\rbrace$. As $g$ is even and $\Omega$ and $\sigma(\Omega)$ are cycles for $g$, for each $\omega \in \Omega$, the points in $\lbrace \omega \rbrace^{\sym} \setminus \left\lbrace \omega, \sigma(\omega) \right\rbrace$ are strictly preperiodic for $g$, and hence the set $\lbrace \omega \rbrace^{\sym} \cap \Omega$ equals $\left\lbrace \omega, \sigma(\omega) \right\rbrace$ or $\lbrace \omega \rbrace$ according to whether $\Omega = \sigma(\Omega)$ or $\Omega \cap \sigma(\Omega) = \varnothing$. Similarly, as $g$ is real and even and $E$ and $\Lambda$ are symmetric with respect to the real axis, each point in $\left( C \cup C^{\prime} \right)^{\sym}$ is either strictly preperiodic for $g$ or it is periodic for $g$, its cycle is contained in $E$ and its multiplier lies in $\Lambda$, and hence $\Omega \subset \mathbb{C}^{\locus} \setminus \left( C \cup C^{\prime} \right)^{\sym}$. Moreover, $A \subseteq E \cap g^{-1}(E)$ and $\Omega \cap D(0, R) \neq \varnothing$. Therefore, there exists a neighborhood $V$ of $\id_{\Omega}$ in $\boldsymbol{Z}$ such that, for each $\boldsymbol{\zeta} \in V$,
\begin{itemize}
\item $\boldsymbol{\zeta}(\omega) \in \mathbb{C}^{\locus} \setminus A$ for all $\omega \in \Omega \setminus \left( E \cap g^{-1}(E) \right)$,
\item $\Omega_{\boldsymbol{\zeta}} \subset \mathbb{C}^{\locus} \setminus \left( C \cup C^{\prime} \right)^{\sym}$,
\item $\left\lbrace \boldsymbol{\zeta}(\omega) \right\rbrace^{\sym} \cap \Omega_{\boldsymbol{\zeta}} = \begin{cases} \left\lbrace \boldsymbol{\zeta}(\omega), \boldsymbol{\zeta}\left( \sigma(\omega) \right) \right\rbrace & \text{if } \Omega = \sigma(\Omega)\\ \left\lbrace \boldsymbol{\zeta}(\omega) \right\rbrace & \text{if } \Omega \cap \sigma(\Omega) = \varnothing \end{cases}$ for all $\omega \in \Omega$,
\item $\boldsymbol{\zeta}(\omega) \neq \boldsymbol{\zeta}\left( \omega^{\prime} \right)$ for all distinct $\omega, \omega^{\prime} \in \Omega$,
\item $\Omega_{\boldsymbol{\zeta}} \cap D(0, R) \neq \varnothing$.
\end{itemize}
For $\boldsymbol{\zeta} \in V$, define \[ g_{\boldsymbol{\zeta}} = g +\alpha \cdot \sum_{\omega \in \Omega \setminus \left( E \cap g^{-1}(E) \right)} \phi_{\boldsymbol{\zeta}, \omega} \in \mathcal{F}\left( g, A \cup C \cup C^{\prime} \right) \, \text{,} \] where \[ \alpha = \begin{cases} \frac{1}{2} & \text{if } \Omega \subset \mathbb{C} \setminus (\mathbb{R} \cup i \mathbb{R}) \text{ and } \Omega = \sigma(\Omega)\\ 1 & \text{if } \Omega \subset \mathbb{R} \text{ or } \Omega \cap \sigma(\Omega) = \varnothing \end{cases} \] and, for every $\omega \in \Omega \setminus \left( E \cap g^{-1}(E) \right)$, \[ \phi_{\boldsymbol{\zeta}, \omega} = P_{A \cup \left( C \cup C^{\prime} \right)^{\sym} \cup \left( \Omega_{\boldsymbol{\zeta}}^{\sym} \setminus \left\lbrace \boldsymbol{\zeta}(\omega) \right\rbrace^{\sym} \right), \boldsymbol{\zeta}(\omega), \boldsymbol{\zeta}\left( g(\omega) \right) -g\left( \boldsymbol{\zeta}(\omega) \right)}^{\locus} \, \text{.} \] For each $\boldsymbol{\zeta} \in V$, we have $g_{\boldsymbol{\zeta}}\left( \boldsymbol{\zeta}(\omega) \right) = \boldsymbol{\zeta}\left( g(\omega) \right)$ for all $\omega \in \Omega$, and hence $\Omega_{\boldsymbol{\zeta}}$ is a cycle for $g_{\boldsymbol{\zeta}}$. Note that $\lim\limits_{\boldsymbol{\zeta} \rightarrow \id_{\Omega}} g_{\boldsymbol{\zeta}} = g$ in $\mathcal{F}$. Therefore, since $E \cap \mathbb{C}^{\locus}$ is dense in $\mathbb{C}^{\locus}$ and $d_{\mathcal{F}}(f, g) < \varepsilon$, there exists $\boldsymbol{\zeta} \in V \cap E^{\Omega}$ such that $d_{\mathcal{F}}\left( f, g_{\boldsymbol{\zeta}} \right) < \varepsilon$. Now, define \[ \boldsymbol{L} = \left\lbrace \boldsymbol{\lambda} \in \left( \mathbb{C}^{\locus} \right)^{\Omega} : \boldsymbol{\lambda} \circ \sigma \vert_{\Omega \cap \sigma(\Omega)} = \sigma \circ \boldsymbol{\lambda} \vert_{\Omega \cap \sigma(\Omega)} \right\rbrace \, \text{.} \] Note that $0 \in \mathbb{C} \setminus \Omega_{\boldsymbol{\zeta}}$ by the second sentence of the proof. Therefore, for $\boldsymbol{\lambda} \in \boldsymbol{L}$, we can define \[ h_{\boldsymbol{\lambda}} = g_{\boldsymbol{\zeta}} +\alpha \cdot \sum_{\omega \in \Omega} \psi_{\boldsymbol{\lambda}, \omega} \in \mathcal{F}\left( g_{\boldsymbol{\zeta}}, A \cup \Omega_{\boldsymbol{\zeta}}, C \cup C^{\prime} \right) \, \text{,} \] where, for every $\omega \in \Omega$, \[ \psi_{\boldsymbol{\lambda}, \omega} = Q_{\left( C \cup C^{\prime} \right)^{\sym} \cup \left( \left( A \cup \Omega_{\boldsymbol{\zeta}}^{\sym} \right) \setminus \left\lbrace \boldsymbol{\zeta}(\omega) \right\rbrace^{\sym} \right), \boldsymbol{\zeta}(\omega), \boldsymbol{\lambda}(\omega)}^{\locus} \, \text{.} \] Then $\Omega_{\boldsymbol{\zeta}}$ is a cycle for $h_{\boldsymbol{\lambda}}$ and we have $h_{\boldsymbol{\lambda}}^{\prime}\left( \boldsymbol{\zeta}(\omega) \right) = g_{\boldsymbol{\zeta}}^{\prime}\left( \boldsymbol{\zeta}(\omega) \right) +\boldsymbol{\lambda}(\omega)$ for all $\omega \in \Omega$ and all $\boldsymbol{\lambda} \in \boldsymbol{L}$. Moreover, we have $\lim\limits_{\boldsymbol{\lambda} \rightarrow 0} h_{\boldsymbol{\lambda}} = g_{\boldsymbol{\zeta}}$ in $\mathcal{F}$. Therefore, since $\Lambda \cap \mathbb{C}^{\locus}$ is dense in $\mathbb{C}^{\locus}$ and $d_{\mathcal{F}}\left( f, g_{\boldsymbol{\zeta}} \right) < \varepsilon$, there exists $\boldsymbol{\lambda} \in \boldsymbol{L}$ such that $d_{\mathcal{F}}\left( f, h_{\boldsymbol{\lambda}} \right) < \varepsilon$ and the multiplier of $h_{\boldsymbol{\lambda}}$ at $\Omega_{\boldsymbol{\zeta}}$ lies in $\Lambda$. Setting $h = h_{\boldsymbol{\lambda}}$, we have $n_{h} > n_{g}$. Thus, the lemma is proved.
\end{proof}

\subsection{Proof of the theorem}

Finally, let us combine here Lemmas~\ref{lemma:realEvenImage}, \ref{lemma:realEvenPreim} and~\ref{lemma:realEvenCycle} in order to prove Theorem~\ref{theorem:realEven}. Our proof only differs from that of Theorem~\ref{theorem:entire} when $0 \in E$. In this case, we shall also control the orbit of $0$ to apply Lemma~\ref{lemma:realEvenCycle}.

\begin{proof}[Proof of Theorem~\ref{theorem:realEven}]
In the case where $0 \in \mathbb{C} \setminus E$, the proof is identical to that of Theorem~\ref{theorem:entire}, by using Lemmas~\ref{lemma:realEvenImage}, \ref{lemma:realEvenPreim} and~\ref{lemma:realEvenCycle} instead of Lemmas~\ref{lemma:entireImage}, \ref{lemma:entirePreim} and~\ref{lemma:entireCycle}.

Thus, from now on, assume that $0 \in E$. Suppose that $f_{0} \in \mathcal{F}$ and $\varepsilon \in \mathbb{R}_{> 0}$. We shall show that there exists $f \in \mathcal{F}$ such that
\begin{itemize}
\item $d_{\mathcal{F}}\left( f_{0}, f \right) < \varepsilon$,
\item $f^{-1}(E) = E$,
\item the cycles for $f$ are all contained in $E$ and their multipliers all lie in $\Lambda$.
\end{itemize}
Since $\mathcal{E}$ is a closed subset of $\mathcal{F}$ with empty interior, replacing $f_{0}$ and $\varepsilon$ if necessary, we may assume that $f_{0} \in \mathcal{F} \setminus \mathcal{E}$ and $\varepsilon \in \left( 0, \dist\left( f_{0}, \mathcal{E} \right) \right)$. Write $E = \left\lbrace e_{j} : j \in \mathbb{Z}_{\geq 1} \right\rbrace$, with $e_{j} \neq e_{k}$ for all distinct $j, k \in \mathbb{Z}_{\geq 1}$. For $n \in \mathbb{Z}_{\geq 1}$, define $E_{n} = \left\lbrace e_{1}, \dotsc, e_{n} \right\rbrace$. Let us show that there exists a sequence $\left( f_{n} \right)_{n \geq 0}$ of elements of $\mathcal{F}$ such that, for every $n \geq 1$,
\begin{enumerate}
\item\label{item:realEvenDist} $d_{\mathcal{F}}\left( f_{n -1}, f_{n} \right) < \frac{\varepsilon}{2^{n}}$,
\item\label{item:realEvenEqual} $f_{n} \in \mathcal{F}\left( f_{n -1}, A_{n -1}, C_{n -1} \right)$ (see the definitions below),
\item\label{item:realEvenCrit} $f_{n}^{\circ n}(0) \in E \setminus \left( A_{n} \cup C_{n} \right)^{\sym}$,
\item\label{item:realEvenImage} $f_{n}\left( e_{n} \right) \in E$,
\item\label{item:realEvenPreim} $f_{n}^{-1}\left( E_{n} \right) \cap D(0, n) \subset E$,
\item\label{item:realEvenCycle} $f_{n}$ has the property $\left( \Gamma_{n, n} \right)$,
\end{enumerate}
where, for every $n \geq 0$, \[ A_{n} = E_{n} \cup \left( f_{n}^{-1}\left( E_{n} \right) \cap D(0, n) \right) \cup \left\lbrace 0, f_{n}(0), \dotsc, f_{n}^{\circ (n -1)}(0) \right\rbrace \] and $C_{n}$ denotes the union of the cycles for $f_{n}$ with period at most $n$ that intersect $D(0, n)$, with $A_{0} = C_{0} = \varnothing$ by convention. Let us proceed by recursion. Suppose that $n \in \mathbb{Z}_{\geq 0}$ and $f_{1}, \dotsc, f_{n} \in \mathcal{F}$ satisfy these conditions~\eqref{item:realEvenDist}--\eqref{item:realEvenCycle}, and let us prove the existence of $f_{n +1} \in \mathcal{F}$ that satisfies the same conditions. Define \[ B_{n, 0} = A_{n} \cup C_{n} \, \text{,} \quad g_{n, 0} = f_{n} \quad \text{and} \quad B_{n, 1} = A_{n} \cup C_{n} \cup \left\lbrace e_{n +1}, f_{n}^{\circ n}(0) \right\rbrace \, \text{.} \] We have $f_{n}^{\circ n}(0) \in \mathbb{C} \setminus B_{n, 0}^{\sym}$ by the condition~\eqref{item:realEvenCrit}. Therefore, by Lemma~\ref{lemma:realEvenImage} applied successively with different dense subsets of $\mathbb{C}$, there exist $g_{n, 1}, \dotsc, g_{n, n +2} \in \mathcal{F}$ such that, for every $j \in \lbrace 1, \dotsc, n +2 \rbrace$, \[ g_{n, j} \in \mathcal{F}\left( g_{n, j -1}, B_{n, j -1} \right) \, \text{,} \quad d_{\mathcal{F}}\left( g_{n, j -1}, g_{n, j} \right) < \frac{\varepsilon}{(n +2) \cdot 2^{n +3}} \] and \[ g_{n, j} \circ \dotsb \circ g_{n, 1}\left( f_{n}^{\circ n}(0) \right) \in E \setminus B_{n, j}^{\sym} \, \text{,} \] where, for every $j \in \lbrace 2, \dotsc, n +2 \rbrace$, \[ B_{n, j} = B_{n, j -1} \cup \left\lbrace g_{n, j -1} \circ \dotsb \circ g_{n, 1}\left( f_{n}^{\circ n}(0) \right) \right\rbrace \, \text{.} \] Set $g_{n} = g_{n, n +2}$. Note that, for every $j \in \lbrace 1, \dotsc, n +2 \rbrace$, \[ g_{n}^{\circ (n +j)}(0) = g_{n, j} \circ \dotsb \circ g_{n, 1}\left( f_{n}^{\circ n}(0) \right) \] and \[ B_{n, j} = A_{n} \cup C_{n} \cup \left\lbrace e_{n +1} \right\rbrace \cup \left\lbrace g_{n}^{\circ n}(0), \dotsc, g_{n}^{\circ (n +j -1)}(0) \right\rbrace \, \text{.} \] In particular, we have
\begin{enumerate}[label=\textup{(\roman*)}, ref=\textup{\roman*}]
\item\label{item:realEvenDistBis} $d_{\mathcal{F}}\left( f_{n}, g_{n} \right) < \frac{\varepsilon}{2^{n +3}}$,
\item\label{item:realEvenEqualBis} $g_{n} \in \mathcal{F}\left( f_{n}, A_{n} \cup C_{n} \right)$,
\item\label{item:realEvenAll} $g_{n}^{\circ (n +j)}(0) \in E$ for all $j \in \lbrace 1, \dotsc, n +2 \rbrace$,
\item\label{item:realEven(n+1)} $g_{n}^{\circ (n +1)}(0) \in E \setminus \left( A_{n} \cup C_{n} \cup \left\lbrace e_{n +1}, f_{n}^{\circ n}(0) \right\rbrace \right)^{\sym}$,
\item\label{item:realEven(n+2)} $g_{n}^{\circ (n +2)}(0) \in E \setminus E_{n +1}$,
\item\label{item:realEvenNoPer} $g_{n}^{\circ (n +j)}(0) \neq \pm g_{n}^{\circ (n +1)}(0)$ for all $j \in \lbrace 2, \dotsc, n +2 \rbrace$.
\end{enumerate}
Now, define \[ A_{n}^{\prime} = A_{n} \cup \left\lbrace g_{n}^{\circ n}(0), \dotsc, g_{n}^{\circ (2 n +1)}(0) \right\rbrace \, \text{.} \] Then we have $A_{n}^{\prime} \cup C_{n} \subset g_{n}^{-1}(E)$ by the conditions~\eqref{item:realEvenEqual}, \eqref{item:realEvenCrit}, \eqref{item:realEvenImage}, \eqref{item:realEvenCycle}, \eqref{item:realEvenEqualBis} and~\eqref{item:realEvenAll}. Therefore, by Lemma~\ref{lemma:realEvenImage}, there exists $h_{n} \in \mathcal{F}$ such that \[ h_{n} \in \mathcal{F}\left( g_{n}, A_{n}^{\prime} \cup C_{n} \right) \, \text{,} \quad d_{\mathcal{F}}\left( g_{n}, h_{n} \right) < \frac{\varepsilon}{2^{n +3}} \, \text{,} \quad h_{n}\left( e_{n +1} \right) \in E \, \text{.} \] Note that $h_{n} \in \mathcal{F} \setminus \mathcal{E}$ since $d_{\mathcal{F}}\left( f_{0}, h_{n} \right) < \varepsilon$ by the conditions~\eqref{item:realEvenDist} and~\eqref{item:realEvenDistBis}. Define \[ A_{n}^{\prime \prime} = A_{n}^{\prime} \cup \left\lbrace e_{n +1} \right\rbrace = A_{n} \cup \left\lbrace e_{n +1} \right\rbrace \cup \left\lbrace g_{n}^{\circ n}(0), \dotsc, g_{n}^{\circ (2 n +1)}(0) \right\rbrace \, \text{.} \] Then $A_{n}^{\prime \prime} \subseteq E \cap h_{n}^{-1}(E)$ by the conditions~\eqref{item:realEvenEqual}--\eqref{item:realEvenPreim}, \eqref{item:realEvenEqualBis} and~\eqref{item:realEvenAll}. Moreover, $C_{n}$ is a union of cycles for $h_{n}$ that are all contained in $E$ and whose multipliers all lie in $\Lambda$ by the conditions~\eqref{item:realEvenCycle} and~\eqref{item:realEvenEqualBis}. Also note that $0$ is not periodic for $h_{n}$ with period at most $n +1$ since $h_{n}^{\circ (n +1)}(0) \neq h_{n}^{\circ j}(0)$ for all $j \in \lbrace 0, \dotsc, n \rbrace$ by the conditions~\eqref{item:realEvenEqualBis} and~\eqref{item:realEven(n+1)}. Therefore, by Lemma~\ref{lemma:realEvenCycle}, there exists $k_{n} \in \mathcal{F}$ such that \[ k_{n} \in \mathcal{F}\left( h_{n}, A_{n}^{\prime \prime}, C_{n} \right) \, \text{,} \quad d_{\mathcal{F}}\left( h_{n}, k_{n} \right) < \frac{\varepsilon}{2^{n +3}} \, \text{,} \quad k_{n} \text{ has the property } \left( \Gamma_{n +1, n +2} \right) \, \text{.} \] Note that $k_{n} \in \mathcal{F} \setminus \mathcal{E}$ since $d_{\mathcal{F}}\left( f_{0}, k_{n} \right) < \varepsilon$ by the conditions~\eqref{item:realEvenDist} and~\eqref{item:realEvenDistBis}. It follows that $k_{n}^{\circ j} \neq \id_{\mathbb{C}}$ for all $j \in \lbrace 1, \dotsc, n +1 \rbrace$, and hence there exists $R_{n} \in (n +1, n +2)$ such that $\partial D\left( 0, R_{n} \right)$ contains no periodic point for $k_{n}$ with period at most $n +1$. Now, denote by $X_{n}$ the union of the cycles for $k_{n}$ with period at most $n +1$ that intersect $D\left( 0, R_{n} \right)$. Then the cycles for $k_{n}$ in $X_{n}$ are all contained in $E$ and their multipliers all lie in $\Lambda$. Denote by $N_{n} \in \mathbb{Z}_{\geq 0}$ the number of periodic points for $k_{n}$ in $D\left( 0, R_{n} \right)$ with period at most $n +1$, counting multiplicities. Since $\Lambda \subseteq \mathbb{C} \setminus \lbrace 1 \rbrace$, the elements of $X_{n}$ are all simple periodic points for $k_{n}$, and hence $N_{n}$ equals the cardinality of $X_{n} \cap D\left( 0, R_{n} \right)$. Moreover, since the topology of $\mathcal{F}$ is finer than the topology of local uniform convergence, there exists $\varepsilon_{n} \in \left( 0, \frac{\varepsilon}{2^{n +3}} \right)$ such that every $\ell_{n} \in \mathcal{F}$ such that $d_{\mathcal{F}}\left( k_{n}, \ell_{n} \right) < \varepsilon_{n}$ has exactly $N_{n}$ periodic points in $D\left( 0, R_{n} \right)$ with period at most $n +1$, counting multiplicities. Since $A_{n}^{\prime \prime} \cup X_{n} \subset E$, it follows from Lemma~\ref{lemma:realEvenPreim} that there exists $f_{n +1} \in \mathcal{F}$ such that \[ f_{n +1} \in \mathcal{F}\left( k_{n}, A_{n}^{\prime \prime} \cup X_{n} \right) \, \text{,} \quad d_{\mathcal{F}}\left( k_{n}, f_{n +1} \right) < \varepsilon_{n} \, \text{,} \quad f_{n +1}^{-1}\left( E_{n +1} \right) \cap D(0, n +1) \subset E \, \text{.} \] Then $f_{n +1}$ clearly satisfies the conditions~\eqref{item:realEvenDist}, \eqref{item:realEvenEqual}, \eqref{item:realEvenImage} and~\eqref{item:realEvenPreim}. Furthermore, $f_{n +1}$ also satisfies the condition~\eqref{item:realEvenCrit} by the conditions~\eqref{item:realEvenEqualBis}, \eqref{item:realEven(n+1)}, \eqref{item:realEven(n+2)} and~\eqref{item:realEvenNoPer}. Now, note that $X_{n}$ is also a union of cycles for $f_{n +1}$ and $f_{n +1}$ has exactly $N_{n}$ periodic points in $D\left( 0, R_{n} \right)$ with period at most $n +1$, counting multiplicities. As $X_{n} \cap D\left( 0, R_{n} \right)$ has exactly $N_{n}$ elements, it follows that $X_{n}$ is the union of all the cycles for $f_{n +1}$ with period at most $n +1$ that intersect $D\left( 0, R_{n} \right)$. As these are all contained in $E$ and their multipliers all lie in $\Lambda$, the map $f_{n +1}$ also satisfies the condition~\eqref{item:realEvenCycle} since $R_{n} > n +1$. Thus, we have proved the existence of a sequence $\left( f_{n} \right)_{n \geq 0}$ of elements of $\mathcal{F}$ that satisfies the desired conditions. Then the rest of the proof is completely identical to that of Theorem~\ref{theorem:entire}. Thus, the theorem is proved.
\end{proof}

\section{Proof of Theorem~\ref{theorem:quadLike}}
\label{section:quadLike}

Finally, we shall apply Theorem~\ref{theorem:realEven} and Remark~\ref{remark:general} in order to prove Theorem~\ref{theorem:quadLike}.

We define $\mathcal{B}$ to be the real vector space of real and even entire maps $f \colon \mathbb{C} \rightarrow \mathbb{C}$ such that the sequence $\left( f^{(j)}(0) \right)_{j \geq 0}$ of successive derivatives of $f$ at $0$ is bounded, and we equip it with the norm $\lVert . \rVert_{\mathcal{B}}$ defined by \[ \lVert f \rVert_{\mathcal{B}} = \sup_{j \geq 0} \left\lvert f^{(j)}(0) \right\rvert \, \text{.} \] Thus, $\mathcal{B}$ is a Banach space of real and even entire maps that contains all the real and even polynomial maps. Moreover, for every $f \in \mathcal{B}$, we have \[ \forall z \in \mathbb{C}, \, \left\lvert f(z) \right\rvert = \left\lvert \sum_{j = 0}^{+\infty} \frac{f^{(2 j)}(0)}{(2 j)!} z^{2 j} \right\rvert \leq \lVert f \rVert_{\mathcal{B}} \cdot \cosh\left( \lvert z \rvert \right) \, \text{,} \] where $\cosh$ denotes the hyperbolic cosine. In particular, the topology of $\mathcal{B}$ is finer than the topology of local uniform convergence on $\mathbb{C}$.

We also define \[ f_{0} = 10 \cosh -12 \in \mathcal{B} \, \text{.} \]

We shall apply Theorem~\ref{theorem:realEven} to prove the existence of $f \in \mathcal{B}$ close to $f_{0}$ such that $f(\mathbb{Q}) \subseteq \mathbb{Q}$ and the periodic points and multipliers of $f \colon \mathbb{R} \rightarrow \mathbb{R}$ all lie in $\mathbb{Q}$. If $f$ has been chosen close enough to $f_{0}$, then we shall also prove that $f$ is transcendental, $f \colon \mathbb{R} \rightarrow \mathbb{R}$ is convex, $f \colon f^{-1}(\mathbb{D}) \cap \mathbb{D} \rightarrow \mathbb{D}$ is an escaping quadratic-like map and the latter two have the same periodic points. Furthermore, using Remark~\ref{remark:general}, $f$ can be chosen so that the multiplier of $f \colon f^{-1}(\mathbb{D}) \cap \mathbb{D} \rightarrow \mathbb{D}$ at its cycle with period $2$ does not equal the product of its multipliers at its two fixed points, which implies that $f \colon f^{-1}(\mathbb{D}) \cap \mathbb{D} \rightarrow \mathbb{D}$ is not conjugate to an affine escaping quadratic-like map.

\subsection{Convexity}

Let us prove here the following result:

\begin{lemma}
\label{lemma:convex}
Suppose that $f \in \mathcal{B}$ satisfies $\left\lVert f -f_{0} \right\rVert_{\mathcal{B}} < 10$. Then $f$ is transcendental and $f \colon \mathbb{R} \rightarrow \mathbb{R}$ is convex.
\end{lemma}

\begin{proof}
As $f_{0}^{(2 j)}(0) = 10$ for all $j \in \mathbb{Z}_{\geq 1}$ and $\left\lVert f -f_{0} \right\rVert_{\mathcal{B}} < 10$, we have $f^{(2 j)}(0) > 0$ for all $j \in \mathbb{Z}_{\geq 1}$. Therefore, $f$ is transcendental. Moreover, $f \colon \mathbb{R} \rightarrow \mathbb{R}$ is convex as \[ \forall x \in \mathbb{R}, \, f^{\prime \prime}(x) = \sum_{j = 0}^{+\infty} \frac{f^{(2 j +2)}(0)}{(2 j)!} x^{2 j} > 0 \, \text{.} \] Thus, the lemma is proved.
\end{proof}

\subsection{Quadratic-like maps}

Let us show here that any $f \in \mathcal{B}$ sufficiently close to $f_{0}$ induces an escaping quadratic-like map $f \colon f^{-1}(\mathbb{D}) \cap \mathbb{D} \rightarrow \mathbb{D}$.

Let us recall that a \emph{quadratic-like map} is a holomorphic proper map $f \colon V \rightarrow W$ of degree $2$, where $V \Subset W$ are nonempty simply connected open subsets of $\mathbb{C}$. In this situation, it follows from the Riemann--Hurwitz formula that $f \colon V \rightarrow W$ has a unique critical point $\gamma_{f} \in V$. Moreover, if $f\left( \gamma_{f} \right) \in W \setminus V$, then $f \colon f^{-1}(V) \rightarrow V$ is an escaping quadratic-like map. However, note that escaping quadratic-like maps are not quadratic-like maps.

\begin{lemma}
\label{lemma:quadLike1}
The holomorphic map $f_{0} \colon 2 \mathbb{D} \rightarrow f_{0}(2 \mathbb{D})$ is proper of degree $2$ and \[ D(-2, 4) \subset f_{0}(\mathbb{D}) \subset D(-2, 7) \quad \text{and} \quad D(-2, 9) \subset f_{0}(2 \mathbb{D}) \, \text{.} \]
\end{lemma}

\begin{proof}
The map $\cosh$ induces a proper holomorphic map of degree $2$ from the strip \[ S = \left\lbrace z \in \mathbb{C} : \Im(z) \in (-\pi, \pi) \right\rbrace \] to the slit plane $\mathbb{C} \setminus (-\infty, -1]$. Since $\cosh$ is even and $3 \mathbb{D} \subset S$, there is a univalent map $\phi \colon 9 \mathbb{D} \rightarrow \mathbb{C}$ such that $\cosh(z) = \phi\left( z^{2} \right)$ for all $z \in 3 \mathbb{D}$. It follows that the map $f_{0} \colon 2 \mathbb{D} \rightarrow f_{0}(2 \mathbb{D})$ is proper of degree $2$. Moreover, we have $\phi(0) = 1$ and $\phi^{\prime}(0) = \frac{1}{2}$. Therefore, it follows from the Koebe distortion theorem that \[ D(1, r) \subset \phi(\mathbb{D}) \subset D(1, R) \, \text{,} \quad \text{with} \quad r = \frac{\frac{1}{2}}{\left( 1 +\frac{1}{9} \right)^{2}} > \frac{2}{5} \, \text{,} \quad R = \frac{\frac{1}{2}}{\left( 1 -\frac{1}{9} \right)^{2}} < \frac{7}{10} \, \text{,} \] and \[ D(1, s) \subset \phi(4 \mathbb{D}) \, \text{,} \quad \text{with} \quad s = \frac{\frac{4}{2}}{\left( 1 +\frac{4}{9} \right)^{2}} > \frac{9}{10} \, \text{.} \] Thus, the desired inclusions hold, and the lemma is proved.
\end{proof}

\begin{lemma}
\label{lemma:quadLike2}
Suppose that $f \in \mathcal{B}$ satisfies $\left\lVert f -f_{0} \right\rVert_{\mathcal{B}} < \frac{1}{4}$. Then $f \colon \mathbb{D} \rightarrow f(\mathbb{D})$ is a quadratic-like map and $f(0) \in (-\infty, -1)$.
\end{lemma}

\begin{proof}
Set $g = f -f_{0}$. Then, for every $z \in 2 \overline{\mathbb{D}}$, \[ \left\lvert g(z) \right\rvert = \left\lvert \sum_{j = 0}^{+\infty} \frac{g^{(2 j)}(0)}{(2 j)!} z^{2 j} \right\rvert \leq \lVert g \rVert_{\mathcal{B}} \cdot \sum_{j = 0}^{+\infty} \frac{2^{2 j}}{(2 j)!} < \frac{\cosh(2)}{4} < 1 \, \text{.} \] In particular, \[ f(0) = -2 +g(0) \in D(-2, 1) \cap \mathbb{R} \subset (-\infty, -1) \, \text{.} \] Moreover, by Lemma~\ref{lemma:quadLike1} and the argument principle, it follows that \[ \overline{D(-2, 3)} \subset f(\mathbb{D}) \subset D(-2, 8) \] and every element of $D(-2, 8)$ has exactly two preimages in $2 \mathbb{D}$ under $f$, counting multiplicities. Therefore, $\mathbb{D} \Subset f(\mathbb{D})$ and the map $f \colon \mathbb{D} \rightarrow f(\mathbb{D})$ is proper of degree $2$. This completes the proof of the lemma.
\end{proof}

\subsection{Same periodic points}

Let us prove here the following result:

\begin{lemma}
\label{lemma:samePer}
Suppose that $f \in \mathcal{B}$ is such that $f \colon \mathbb{R} \rightarrow \mathbb{R}$ is convex, $f \colon \mathbb{D} \rightarrow f(\mathbb{D})$ is a quadratic-like map and $f(0) \in (-\infty, -1)$. Then $f \colon f^{-1}(\mathbb{D}) \cap \mathbb{D} \rightarrow \mathbb{D}$ is an escaping quadratic-like map and its periodic points coincide with those of $f \colon \mathbb{R} \rightarrow \mathbb{R}$.
\end{lemma}

\begin{proof}
Set $U = f^{-1}(\mathbb{D}) \cap \mathbb{D}$. As $f \colon \mathbb{D} \rightarrow f(\mathbb{D})$ is a quadratic-like map, $U \Subset \mathbb{D}$ and the map $f \colon U \rightarrow \mathbb{D}$ is proper of degree $2$. Moreover, the unique critical point of $f$ in $\mathbb{D}$ is $0$ since $f$ is even, and $f(0) \in \mathbb{C} \setminus \mathbb{D}$ by assumption. Therefore, $f \colon U \rightarrow \mathbb{D}$ is an escaping quadratic-like map. Thus, $U$ has two connected components $U_{-}$ and $U_{+}$ and these are mapped biholomorphically onto $\mathbb{D}$ by $f$. Denote by $g_{\pm} \colon \mathbb{D} \rightarrow U_{\pm}$ the inverse of $f \colon U_{\pm} \rightarrow \mathbb{D}$. We have $U_{+} = -U_{-}$ and $g_{+} = -g_{-}$ because $f$ is even. A standard argument using the fact that the map $g_{\pm}$ is contracting with respect to the Poincar\'{e} metric on $\mathbb{D}$ shows that, for every sign sequence $\boldsymbol{\varepsilon} = \left( \epsilon_{n} \right)_{n \geq 0}$, there is a point $\zeta(\boldsymbol{\varepsilon}) \in \mathbb{D}$ such that \[ \bigcap_{n \geq 0} g_{\epsilon_{0}} \circ \dotsb \circ g_{\epsilon_{n}}(\mathbb{D}) = \left\lbrace \zeta(\boldsymbol{\varepsilon}) \right\rbrace \, \text{.} \] Now, denote by $\sigma$ the shift map, which sends a sign sequence $\left( \epsilon_{n} \right)_{n \geq 0}$ to $\left( \epsilon_{n +1} \right)_{n \geq 0}$. Then the periodic points of $f \colon U \rightarrow \mathbb{D}$ are precisely the points $\zeta(\boldsymbol{\varepsilon})$, with $\boldsymbol{\varepsilon}$ a sign sequence that is periodic for $\sigma$.

Since $f \colon \mathbb{R} \rightarrow \mathbb{R}$ is even and convex, it is increasing on $\mathbb{R}_{\geq 0}$. As $f \colon \mathbb{D} \rightarrow f(\mathbb{D})$ is a quadratic-like map, it follows that $f(1) > 1$. Therefore, $U_{\pm} \cap \mathbb{R} \neq \varnothing$ because we also have $f(0) < -1$ by assumption. It follows that $U_{\pm}$ is symmetric with respect to the real axis and $g_{\pm}$ commutes with complex conjugation. Therefore, for every sign sequence $\boldsymbol{\varepsilon}$, the set $\left\lbrace \zeta(\boldsymbol{\varepsilon}) \right\rbrace$ is also symmetric with respect to the real axis, and hence $\zeta(\boldsymbol{\varepsilon}) \in \mathbb{R}$. In particular, every periodic point of $f \colon U \rightarrow \mathbb{D}$ is also a periodic point of $f \colon \mathbb{R} \rightarrow \mathbb{R}$.

Finally, note that $f(x) > x$ for all $x \in [1, +\infty)$ as $f \colon \mathbb{R} \rightarrow \mathbb{R}$ is convex, $f(0) < 0$ and $f(1) > 1$. It follows that the map $f \colon \mathbb{R} \rightarrow \mathbb{R}$ has no periodic point in $[1, +\infty)$. Therefore, the periodic points of $f \colon \mathbb{R} \rightarrow \mathbb{R}$ all lie in $(-1, 1) \subset \mathbb{D}$ because $f$ is even. Thus, every periodic point of $f \colon \mathbb{R} \rightarrow \mathbb{R}$ is also a periodic point of $f \colon U \rightarrow \mathbb{D}$, and the lemma is proved.
\end{proof}

\subsection{Proof of the theorem}

Finally, let us combine here Theorem~\ref{theorem:realEven}, Remark~\ref{remark:general} and Lemmas~\ref{lemma:convex}, \ref{lemma:quadLike2} and~\ref{lemma:samePer} in order to prove Theorem~\ref{theorem:quadLike}.

\begin{proof}[Proof of Theorem~\ref{theorem:quadLike}]
Define \[ \mathbb{Q}_{0} = \left\lbrace \frac{p}{q} \in \mathbb{Q} : p \text{ odd}, \, q \text{ even} \right\rbrace \quad \text{and} \quad \mathbb{Q}_{1} = \left\lbrace \frac{p}{q} \in \mathbb{Q} : q \text{ odd} \right\rbrace \] to be the sets of rational numbers with even and odd denominators. Also define \[ E = \mathbb{Q}(i) \, \text{,} \quad \Lambda = \mathbb{Q}_{0} \cup (\mathbb{C} \setminus \mathbb{R}) \quad \text{and} \quad \Lambda^{\prime} = \mathbb{Q}_{1} \cup (\mathbb{C} \setminus \mathbb{R}) \, \text{.} \] By Theorem~\ref{theorem:realEven} and Remark~\ref{remark:general}, there exists $f \in \mathcal{B}$ such that
\begin{enumerate}
\item\label{item:quadLikeDist} $\left\lVert f -f_{0} \right\rVert_{\mathcal{B}} < \frac{1}{4}$,
\item\label{item:quadLikePreim} $f^{-1}(E) = E$,
\item\label{item:quadLikePer} the periodic points of $f$ all lie in $E$,
\item\label{item:quadLikeMult} the multipliers of $f$ at its cycles with period different from $2$ all lie in $\Lambda$,
\item\label{item:quadLikeMultBis} the multipliers of $f$ at its cycles with period $2$ all lie in $\Lambda^{\prime}$.
\end{enumerate}
By the condition~\eqref{item:quadLikeDist} and Lemmas~\ref{lemma:convex}, \ref{lemma:quadLike2} and~\ref{lemma:samePer}, the map $f$ is transcendental, the map $f \colon \mathbb{R} \rightarrow \mathbb{R}$ is convex and $f \colon f^{-1}(\mathbb{D}) \cap \mathbb{D} \rightarrow \mathbb{D}$ is an escaping quadratic-like map whose cycles coincide with those of $f \colon \mathbb{R} \rightarrow \mathbb{R}$. We have $f(\mathbb{Q}) \subset E \cap \mathbb{R} = \mathbb{Q}$ by the condition~\eqref{item:quadLikePreim}. Now, by the condition~\eqref{item:quadLikePer}, the periodic points of $f \colon \mathbb{R} \rightarrow \mathbb{R}$ all lie in $E \cap \mathbb{R} = \mathbb{Q}$. Moreover, by the conditions~\eqref{item:quadLikeMult} and~\eqref{item:quadLikeMultBis}, the multipliers of $f \colon \mathbb{R} \rightarrow \mathbb{R}$ all lie in $\Lambda \cap \mathbb{R} = \mathbb{Q}_{0}$, apart from that at its unique cycle with period $2$ which lies in $\Lambda^{\prime} \cap \mathbb{R} = \mathbb{Q}_{1}$. In particular, the multipliers of $f \colon \mathbb{R} \rightarrow \mathbb{R}$ all lie in $\mathbb{Q}$. Finally, the product of the multipliers of $f \colon f^{-1}(\mathbb{D}) \cap \mathbb{D} \rightarrow \mathbb{D}$ at its two fixed points lies in $\mathbb{Q}_{0}$, as both of them lie in $\mathbb{Q}_{0}$, and hence it differs from the multiplier at the cycle with period $2$. It follows that $f \colon f^{-1}(\mathbb{D}) \cap \mathbb{D} \rightarrow \mathbb{D}$ is not conjugate to an affine escaping quadratic-like map. Thus, the theorem is proved.
\end{proof}

\providecommand{\bysame}{\leavevmode\hbox to3em{\hrulefill}\thinspace}
\providecommand{\MR}{\relax\ifhmode\unskip\space\fi MR }
\providecommand{\MRhref}[2]{%
	\href{http://www.ams.org/mathscinet-getitem?mr=#1}{#2}
}
\providecommand{\href}[2]{#2}

\end{document}